\documentclass[11pt]{amsart}
\usepackage{amsmath,mathtools}
\usepackage{amsmath, amsthm, amssymb, amsfonts, enumerate}
\usepackage[colorlinks=true,linkcolor=blue,urlcolor=blue]{hyperref}
\usepackage{dsfont}
\usepackage{color}
\usepackage{geometry}
\usepackage{todonotes}
\usepackage{epstopdf}
\usepackage{bbm}
\usepackage{geometry}
\usepackage[utf8]{inputenc}
\usepackage{bm}
\usepackage{amsfonts}
\usepackage{amsfonts}
\usepackage{textcomp}
\usepackage{amssymb}
\usepackage{float}
\usepackage{tikz}
\usepackage{epsfig}
\usepackage{amsmath}
\usepackage[english]{babel}
\usepackage{a4}
\usepackage{enumerate}
\usepackage{tikz,pgf,xspace}
\usepackage[all]{xy}
\usepackage{amsmath}
\newcommand{\stkout}[1]{\ifmmode\text{\sout{\ensuremath{#1}}}\else\sout{#1}\fi}

\newtheorem{theorem}{Theorem}[section]
\newtheorem{remark}[theorem]{Remark}

\newtheorem{proposition}[theorem]{Proposition}

\def \E{\mathsf{E}}

\def \R{\mathbb{R}}

\DeclareMathOperator*{\argmin}{argmin}

\definecolor{red}{rgb}{1.0,0.0,0.0}

\definecolor{blu}{rgb}{0.0,0.0,1.0}

\definecolor{gre}{rgb}{0.03,0.50,0.03}

\usepackage{color}
\definecolor{red-}{rgb}{1.0,0.0,0.0}
\definecolor{grey}{rgb}{0.6, 0.6, 0.6}
\definecolor{brown}{rgb}{0.5,0.2,0.0}
\definecolor{brown-}{rgb}{0.0,0.1,1.0}
\definecolor{green-}{rgb}{0.0, 0.6, 0.0}
\definecolor{gold}{rgb}{0.8,0.7,0.0}
\definecolor{black}{rgb}{0.0,0.0,0.0}
\definecolor{DarkGreen}{rgb}{0.0,0.3,0.2}
\definecolor{LightGreen}{rgb}{0.8,1.0, 0.8}
\definecolor{yellow}{rgb}{0.9,0.9,0.0}
\definecolor{blue-}{rgb}{0.0,0.1,1.0}

\title[Irreversible reinsurance with fixed cost]{Irreversible reinsurance: Minimization of Capital Injections in Presence of a Fixed Cost}
\author[Federico]{Salvatore Federico}
\author[Ferrari]{Giorgio Ferrari}
\author[Torrente]{Maria Laura Torrente}

\address{S.~Federico: Dipartimento di Matematica, Universit\`a di Bologna, Piazza di Porta S.\ Donato 5, 40126, Bologna, Italy}
\email{\href{mailto:s.federico@unibo.it}{s.federico@unibo.it}}
\address{G.~Ferrari: Center for Mathematical Economics (IMW), Bielefeld University, Universit\"atsstrasse 25, 33615, Bielefeld, Germany}
\email{\href{mailto:giorgio.ferrari@uni-bielefeld.de}{giorgio.ferrari@uni-bielefeld.de}}
\address{M.L.~Torrente: Dipartimento di Economia, Universit\`a di Genova, Via Vivaldi 5, 16126 Genova, Italy}
\email{\href{mailto:marialaura.torrente@economia.unige.it}{marialaura.torrente@economia.unige.it}}
\date{\today}

\numberwithin{equation}{section}

\begin{document}

\begin{abstract} 
We propose a model in which, in exchange to the payment of a fixed transaction cost, an insurance company can choose the retention level as well as the time at which subscribing a perpetual reinsurance contract. The surplus process of the insurance company evolves according to the diffusive approximation of the Cram\'er-Lundberg model, claims arrive at a fixed constant rate, and the distribution of their sizes is general. Furthermore, we do not specify any particular functional form of the retention level. The aim of the company is to take actions in order to minimize the sum of the expected value of the total discounted flow of capital injections needed to avoid bankruptcy and of the fixed activation cost of the reinsurance contract. We provide an explicit solution to this problem, which involves the resolution of a static nonlinear optimization problem and of an optimal stopping problem for a reflected diffusion. We then illustrate the theoretical results in the case of proportional and excess-of-loss reinsurance, by providing a numerical study of the dependency of the optimal solution with respect to the model's parameters.
\end{abstract}

\maketitle

\smallskip


\smallskip

{\textbf{MSC2010 subject classification}}: 97M30, 91B30, 60G40, 49L20.
  
\smallskip



\section{Introduction}
\label{introduction}

Reinsurance contracts usually run for long time periods (at least for longer than the typical maturity of financial contracts) and are exposed to high frictional costs. As a result, reinsurance negotiations are costly, lengthy, and can be thought of as irreversible, cf.\ \cite{Changetal}. As noticed by \cite{CumminsGeman}, it is indeed the case that ``although reinsurance, in principle, is reversible, in practice reversing a reinsurance transaction exposes the insurer to relatively high transaction costs as well as additional charges to protect the reinsurer against adverse selection.'' Furthermore, many external factors can interfere with changes in the reinsurance contracts. It is recent news that ``2023's renegotiation of reinsurance policies has been the most challenging in years as reinsurers respond to pressure from spiralling inflation and large losses from natural catastrophes, as well as the fallout from Russia's invasion of Ukraine'' (cf.\ Ian Smith, Insurance Correspondent of the ``Financial Times'', January 3 2023\footnote{\url{https://www.ft.com/content/f5f9d450-c539-47a7-bc5c-44a8db57e74e}}). 

Optimal reinsurance decisions are typically formulated in terms of regular control problems, thus neglecting the aforementioned irreversibility feature. Given the vastity of the related literature, we refrain here form providing a list of references (that would necessarily result in being not exhaustive) and we simply refer to the discussion in Chapter 2 of \cite{Schmidli2008} or in Chapter 11 of \cite{AsmussenSteffensen} for models and solutions. However, in the last decade the actuarial literature has started experiencing models of optimal irreversible reinsurance. In \cite{BrachettaCeci2021} it is investigated an optimal reinsurance problem under fixed cost, for an insurance company aiming at maximizing exponential expected utility at terminal time. The problem of optimal reinsurance negotiations with implementation delay and fixed cost is considered in \cite{EgamiYoung}, the optimal timing for the activation of an excess-of-loss reinsurance with fixed costs is studied in \cite{LiZhouYao}, while the presence of additional proportional transaction costs for a company minimizing the ruin probability is treated in \cite{Lietal}. Finally, a singular stochastic control model for the optimal sequential adjustment of reinsurance contracts has been recently formulated in \cite{Yanetal}.

Our paper contributes to that bunch of literature by proposing a model in which, in exchange to the payment of a fixed transaction cost, an insurance company can choose the retention level as well as the time at which subscribing a (perpetual) reinsurance contract. Assuming that investors inject capital to avoid bankruptcy of the company, the insurance company aims at minimizing the sum of the expected value of the discounted fixed activation cost and of the cumulative discounted flow of capital injections. Capital injection models have been introduced by Dickson and Waters in \cite{DiksonWaters}. Therein, starting from the observation that ruin occurs almost surely when the company pays dividends by following the optimal strategy of the de Finetti's problem, a model has been suggested in which the shareholders are obliged to inject capital in order to avoid bankruptcy. We also refer to \cite{Ferrari17}, \cite{FerrariSchuhmann}, \cite{LokkaZervos}, \cite{ScheerSchmidli2011}, \cite{Schmidli2016}, \cite{Yang} and reference therein for works related to the optimal dividends' distribution in presence of capital injections (see also \cite{Decamps} and \cite{Zhou} for the case of impulsive injections of capital). In the context of optimal reinsurance problems, the employment of the cumulative discounted flow of capital injections as a risk measure alternative to the ruin probability has firstly been proposed in \cite{EisenbergSchmidli2009}, and later also used in \cite{EisenbergFabrykowskiSchmeck} and \cite{EisenbergSchmidli2011b}. As a matter of fact, the use of the ruin probability as measure of risk presents drawbacks: first of all, it is not a coherent risk measure, this potentially leading to decisions that are not economically sounded; second of all, it does not provide information about neither the time of ruin nor the severity of ruin. In order to define a unified framework for the evaluation of a variety of risk quantities, and in particular to give indications about the deficit at ruin and the time of ruin, Gerber and Shiu proposed in \cite{GerberShiu} the so-called expected discounted penalty function -- also known as Gerber-Shiu function -- of which the capital injection criterion represents an example (see, e.g., Section 2.4.3 in \cite{HeKawaiShimizuYamazaki2023} or Chapter 4 in \cite{SilvestrovMartinLoef}).

In this paper, we assume that the surplus process of an insurance company evolves according to the diffusive approximation of the Cram\'er-Lundberg model. Claims arrive at a fixed constant rate and the distribution of their sizes is general. Furthermore, we do not specify any form of the retention level, which is simply assumed to be a continuous function, non decreasing with respect to the reinsurance parameter. The company can choose the time $\tau$ at which buying reinsurance and the desired retention level, which will then be kept from time $\tau$ on. Those once-for-all actions involve a fixed cost, which is immediately withdrawn from the company's surplus at time $\tau$. The aim of the company is to take actions in such a way that the total discounted costs of capital injection and of the reinsurance contract are minimized. We provide an explicit solution to this problem which we show can be solved via a two-step procedure (see also \cite{BrachettaCeci2021} and \cite{Lietal}, among others). We first solve for the optimal retention level, which is uniquely identified through the solution to a nonlinear algebraic equation. Then, given the optimal retention level, we look for the optimal time at which it is worth activating the irreversible reinsurance contract. This turns out to be given as the solution to a one-dimensional optimal stopping problem for a reflected drifted Brownian motion. We use the classical guess-and-verify approach by determining a smooth solution to the corresponding variational inequality with Neumann boundary condition and then by verifying the actual optimality of the candidate policy. It is worth noticing that, given the reflecting condition of the surplus process at zero, the verification argument requires quite some technical work in order to check that the variational inequality is indeed satisfied by the candidate value function (see the proofs of Proposition \ref{prop1} and of Theorem \ref{prop2} below).
We find that a barrier-strategy is optimal and that reinsurance should be bought when the insurance company's surplus process is sufficiently large, in particular larger than an endogenously determined trigger level (free boundary) that depends on the model's parameters. Interestingly, we observe that the solution to our problem is consistent with that of \cite{EisenbergSchmidli2009}, where, given the absence of a fixed transaction cost, reinsurance is bought immediately. Namely, the optimal retention level in our model is the same as that in \cite{EisenbergSchmidli2009}, and, when the fixed cost $K\downarrow 0$, the free boundary converge to zero as well, implying that immediate reinsurance is in fact optimal.

We finally illustrate our results in the two relevant cases of proportional and excess-of-loss reinsurance, when the distribution of the claims' sizes are Exponential or Pareto with parameters $(\zeta,\alpha)$, for some $\alpha>2$ and $\zeta>0$. We solve numerically the equations that uniquely determine the optimal retention level and the free boundary and we study the dependency of those quantities with respect to relevant model's parameters. We observe that both the optimal retention level and the free boundary exhibit a monotonic behavior with respect to the considered parameters and we provide explanations of these findings. Furthermore, we show (for fixed values of the model's parameters) that, when the claim's size is exponentially distributed, the value function one has in the case of proportional reinsurance is smaller than the one related to an excess-of-loss reinsurance, while no uniform comparison can be made in the case of Pareto-distributed claim's size. 

The rest of the paper is organized as follows. Section \ref{sec:setting} presents the problem, which is then solved in Section \ref{sec:sol}. Section \ref{sec:applications} illustrates numerically the theoretical findings in the case of proportional and excess-of-loss reinsurance, while a final appendix collects most of the technical proofs of the paper.

\section{Problem Formulation}
\label{sec:setting}

Let $(\Omega, \mathcal F, \mathbb F:= (\mathcal F_t)_{t \ge 0}, \mathbb P)$ be a complete probability space, rich enough to accommodate a one-dimensional
$\mathbb F$-Brownian motion $(W_t)_{t\ge 0}$ and an independent square-integrable random variable $Z$, taking values in $\mathcal Z \subset \mathbb R_+$,
and with law $\nu_Z$ under $\mathbb P$.
Within this probabilistic setting, we consider the unaffected surplus process $(\widehat{X}_{t})_{t\geq 0}$ of an insurance company, with initial value $\widehat{X}_{0}=x>0$, evolving through the diffusion approximation of the classical Cram\'er-Lundberg model (see, e.g., Appendix D in \cite{Schmidli2008} or Section 8 in Chapter IV of \cite{AsmussenSteffensen})
\begin{eqnarray}\label{dinamo}
\widehat{X}_t^{x}= x+   \lambda \eta\mu t + \sigma \sqrt{\lambda } W_t, \quad t \geq 0. 
\end{eqnarray}
Here, $\mu:=\int_{\mathcal Z} z \nu_Z(dz)>0$ and  $\sigma^2:=\int_{\mathcal Z} z^2 \nu_Z(dz)>0$ are, respectively,  the mean and the second moment of the generic claim size $Z$, $\lambda$ is the arrival time parameter of the claims, $\eta$ is the safety loading.

In order to avoid bankruptcy, investors are asked by the insurance company to inject capital whenever the surplus level attempts to become negative. Assuming that investors are impatient agents, it is clear that those injections of capital are made only when strictly necessary. The cumulative amount of capital injections 
$(I_t)_{t \ge 0}$ will then reflect (\`a la Skorokhod) the surplus process at $x=0$, so that the resulting dynamics are
 \begin{eqnarray}\label{dinamop}
X_t^{x}= x+   \lambda \eta\mu t + \sigma\sqrt{\lambda} W_t+I_{t} = \widehat{X}_t^{x}+I_t, \quad t\geq 0,
\end{eqnarray}
with
$$
I_t=\sup_{0 \le s \le t} \left[-\widehat{X}_s^x\right]^+, \quad t\geq 0.
$$

Within this model, we consider the possibility for the insurer of adopting a reinsurance strategy. More precisely, we  consider a continuous function
$$
r: \mathcal{Z} \times [0,1]\to \mathbb{R}^{+},
$$
which represents the retention level of the insurer --- that is, the part of risk remaining in her charge --- whose value depends on the chosen level $b\in [0,1]$. 
It is assumed $r(z,\cdot)$ is non decreasing and that $b=1$ corresponds to no reinsurance and $b=0$ corresponds to full reinsurance; that is, 
\begin{eqnarray}\label{rz1}
r(z,1)=z, \ \ \ r(z,0)=0.
\end{eqnarray}
Typical examples are the case of proportional reinsurance, for which
\begin{eqnarray}\label{propReins}
r(z,b)=bz,
\end{eqnarray}
and that of the excess-of-loss reinsurance, for which 
\begin{eqnarray}\label{excessOfLossReins}
r(z,b)= z\wedge \left(\frac{b}{1-b}\right).
\end{eqnarray}

\begin{remark}
It is worth noticing that choosing the reinsurance parameter $b \in [0,1]$ allows us to cover the relevant reinsurance models using a unique parametrization within
a unified setting. In the case of excess-of-loss reinsurance (cf.\ \eqref{excessOfLossReins} above), this leads to a deviation from the classical formula of the reinsurance retention level which assumes $b\geq0$ (cf.\ \cite{Schmidli2008}).  
\end{remark}

In our model, we assume that the reinsurance policy is irreversible. This means that at a properly picked $\mathbb F$-stopping time $\tau$ the insurer chooses the
level $b_{\tau}$, which will then be kept from time $\tau$ on. Formally, the reinsurance policy is thus a couple 
$$
a:=(\tau,b_{\tau})\in \mathcal{A}:=\mathcal{T}\times \mathcal{M}_{\mathcal{F}_\tau},$$ where
$$
\mathcal{T}:=\{\tau:\Omega\to\mathbb{R}^{+} \ \mathbb F\mbox{-stopping time}\}, \ \ \ \ \mathcal{M}_{\mathcal{F}_\tau}:=\{B:(\Omega, \mathcal{F}_{\tau})\to[0,1] \ \mbox{measurable}\}.
$$  
To implement the reinsurance strategy $a$, the insurer faces a fixed transaction cost $K\ge 0$ at the time $\tau \in \mathcal T$ at which the 
reinsurance contract is signed. From time $\tau$ on, according to the expected value principle, the insurer pays to the reinsurance company a perpetual premium rate with value
\begin{eqnarray}\label{formula1}
\theta\lambda (\mu- M_1(b_{\tau})), \ \ \ \ \mbox{where} \ M_{1}(b):=\int_{\mathcal Z}r(z,b)\nu_Z(dz),
\end{eqnarray}
for $\theta>\eta$. On the other hand, the risk exposure of the insurance company is reduced, leading to the diffusion coefficient
\begin{eqnarray}\label{formula2}
\lambda M_{2}(b_\tau), \ \ \ \mbox{where} \ \ M_{2}(b):=\int_{\mathcal Z}r(z,b)^2\nu_Z(dz).
\end{eqnarray}
In particular, from \eqref{rz1} it follows that $M_1({1})=\mu$ and $M_2({1})=\sigma^2$.
Consequently, for $t\ge \tau$, the insurer only faces the outflows relative to her part of risk, represented by the retention level $r$. All in all, the dynamics with capital injection of the surplus process after time $\tau$ are
\begin{eqnarray}\label{surplus}
X^{x}_{t}= X^{x}_{\tau^{-}}-K+ \lambda(\theta M_1(b_{\tau}) - (\theta - \eta)\mu)(t-\tau)+ \sqrt{\lambda M_2(b_\tau)}(W_{t}-W_{\tau}) + I_{t}, \ \ \ t\geq \tau,
\end{eqnarray}
where $(I_{t})_{t\ge 0}$ is now such that
$$
I_{t}=\sup_{\tau\leq s\leq t}\left[-\left(X^{x}_{\tau^{-}}-K+ \lambda(\theta M_1(b_{\tau}) - (\theta - \eta)\mu)(t-\tau)+ \sqrt{\lambda M_2(b_\tau)}(W_{t}-W_{\tau}) \right)\right]^{+}, \ \ \ t\geq \tau.
$$
In the sequel, in order to stress the dependency of $I$ on the reinsurance policy and $x$, we shall write $I^{x,a}$, when needed.

Following \cite{EisenbergFabrykowskiSchmeck}, \cite{EisenbergSchmidli2009}, and \cite{EisenbergSchmidli2011}, we assume that the insurance company employs the expected total amount of discounted capital injections as a measure of risk and thus aims at determining an admissible irreversible reinsurance policy $a^*\in\mathcal{A}$ such that
$$a^* \in \argmin_{a\in \mathcal A} \E\left[\int_0^\infty e^{-\rho t} dI_t^{x,a}\right],$$
where $\rho>0$ is a subjective intertemporal discount rate.
For future frequent use we also define
\begin{eqnarray}
\label{eq:defU}
U(x) := \inf_{a\in \mathcal A} \E\left[\int_0^\infty e^{-\rho t} dI_t^{x,a}\right], \quad x\geq 0.
\end{eqnarray}

\section{Solution to the problem}
\label{sec:sol}

In this section, we determine the explicit solution to \eqref{eq:defU}. To accomplish that, we shall first reformulate the problem in an handier way (cf.\ Section \ref{sec:reformulation}), then we shall obtain the optimal level (cf.\ Section \ref{sec:optlevel}) and, finally, the optimal time for reinsurance (cf.\ Section \ref{sec:opttime}).

\subsection{Reformulation of the problem}
\label{sec:reformulation}

In order to obtain an handy representation of $U$, we compute the injection costs associated to a fixed retention parameter $b\in[0,1]$ taken at $t=0$;  
that is, given $y \in \mathbb R$ and $b\in[0,1]$, we calculate
\begin{eqnarray}\label{Gb}
G_b(y):=\E\left[\int_{0}^{\infty}e^{-\rho t}dH_{t}^{y,b}\right],
\end{eqnarray}
where
\begin{eqnarray}\label{Hyb}
H^{y,b}_{t}:=\sup_{0\leq s\leq t}\left[-Y^{y,b}_s\right]^{+}, \quad t\geq 0,
\end{eqnarray}
with
$$
Y^{y,b}_{t}:= y+ \lambda(\theta M_1(b) - (\theta - \eta)\mu)t+ \sqrt{\lambda M_2(b)}\,\widetilde{W}_{t}, \ \ \ t\geq 0,
$$
for another $\mathbb F$-Brownian motion $(\widetilde{W}_t)_{t \ge 0}$.

Following \cite{Shreve}, we know that, when $y\geq 0$, the function $G_b$ is the solution to the differential problem
\begin{eqnarray}\label{eqDiff}
\left\{\begin{array}{lll}
\displaystyle{\frac{1}{2}\lambda M_2(b){G_b''}(y)
+\lambda(\theta M_1(b)- (\theta - \eta)\mu){G_b'}(y)-\rho G_b(y)=0},\\ \\
\displaystyle{{G_b'}(0)=-1}, \ \ \ \ 
\lim_{y \to +\infty}G_b(y)=0.
\end{array}\right.
\end{eqnarray}
It then follows from \eqref{eqDiff} that 
\begin{eqnarray}\label{GbxAnalyticPos}
G_b(y) = -\frac{1}{\gamma^-(b)}e^{\gamma^-(b)y}, \ \ \ \forall y\geq 0,
\end{eqnarray}
where $\gamma^-(b)<0$ is the negative solution to 
 the equation $\Phi(b,\gamma)=0$, with
\begin{eqnarray}\label{pol}
\Phi(b,\gamma):=\frac{1}{2}\lambda M_2(b) \gamma^2+
\lambda\left(\theta M_1(b)- (\theta - \eta)\mu\right)\gamma-\rho, \quad \gamma \in  \mathbb R.
\end{eqnarray}
On the other hand, we have
\begin{eqnarray}\label{GbxAnalyticNeg}
G_b(y)=-y+G_{b}(0)= -y-\frac{1}{\gamma^-(b)}, \ \ \ \forall y< 0.
\end{eqnarray}

With the help of the previously defined quantities (cf.\ \eqref{Gb} and \eqref{Hyb}), an application of the strong Markov property allows us to rewrite $U$ as follows:
\begin{eqnarray*}
U(x)&=& \inf_{a\in \mathcal A} {\E} \left[\int_0^{\tau^{-}} e^{-\rho t} dI_t^{x,a} + \int_{\tau^{-}}^\infty e^{-\rho t} dI_t^{x,a}\right]\\
&=& \inf_{a\in \mathcal A} {\E} \left[\int_0^{\tau^{-}} e^{-\rho t} dI_t^{x,a} +\E\left[\int_{\tau^{-}}^\infty e^{-\rho t} dI_t^{x,a}\ \big|\ \mathcal{F}_{\tau^-}\right]\right]\\
&=& \inf_{a\in \mathcal A} {\E} \left[\int_0^{\infty} e^{-\rho t} dH^{x,1}_t -\int_{\tau^{-}}^{\infty} e^{-\rho t} dH_t^{x,1}+\E\left[\int_{\tau^{-}}^\infty e^{-\rho t} dI_t^{x,a}\ \big|\ \mathcal{F}_{\tau^-}\right]\right]\\
 &=&G_1(x)+  \inf_{a \in \mathcal A} {\E}\left[e^{-\rho\tau}\left(G_{b_\tau}(X_{\tau^{-}}^{x}-K) -G_1(X_{\tau^{-}}^{x})\right)\right]\\
  &=:&G_1(x)+  \inf_{a \in \mathcal A} \mathcal{J}(x,a).
\end{eqnarray*}
Letting 
\begin{eqnarray}\label{OP}
V(x) := \inf_{a \in \mathcal A} \mathcal{J}(x, a),
\end{eqnarray}
where 
\begin{eqnarray}\label{fb}
\mathcal{J}(x,a):= \E\left[e^{-\rho\tau}f_{b_{\tau}}(X^x_{\tau^{-}})\right], \ \ \ \ 
f_{b}(y):= G_b(y-K) - G_1(y),
\end{eqnarray}
we have
\begin{eqnarray}\label{VG1U}
U(x)=G_{{1}}(x) + V(x).
\end{eqnarray}
We now continue our analysis by determining the optimal $a^*=(\tau^*,b^*)$ s.t.\ $V(x) =  \mathcal{J}(x, a^*)$. Clearly, such an $a^*$ will also be optimal for \eqref{eq:defU}.


\subsection{Optimal reinsurance}
\label{sec:optlevel}

Recall that the function $\gamma^-(b)$ has been defined as the negative solution to the equation $\Phi(b,\cdot)=0$, 
with $\Phi$ as in \eqref{pol}. Its explicit expression is
\begin{eqnarray}\label{gammaMeno}
\gamma^-(b) = \displaystyle{-\frac{(\theta M_1(b)-(\theta -\eta)\mu) + \sqrt{(\theta M_1(b)-(\theta -\eta)\mu)^2+2\frac{\rho M_2(b)}{\lambda}}}{M_2(b)}}, 
\ \ \ \ \ \ \ b \in [0,1].
\end{eqnarray}

We denote 
$$
\mathcal{B}^*\ :=\ \textrm{argmin}_{b \in[0,1]}\ \gamma^-(b)\ .
$$
A relevant fact is that 
\begin{eqnarray}\label{minGb}
b^* \in \mathcal{B}^* \quad \Longrightarrow \quad G_{b^*}(x)=\min_{b \in [0,1]}G_b(x), \quad \forall x \ge 0. 
\end{eqnarray}
Indeed,
$$
G_b(x)=H(\gamma^-(b),x),
$$
where $H: \mathbb R^- \times \mathbb R^+ \to \mathbb R$ is defined by 
$$
H(g,x):=-\frac{1}{g} e^{gx}.
$$
Since 
$$
\frac{\partial H}{\partial g}(g,x) =  \frac{1}{g^2}e^{gx}(1-gx)> 0,\ \ \ \ \ \forall (g,x) \in \mathbb R^- \times \mathbb R^+,
$$
we see that $H$ is strictly increasing with respect to the first variable. It follows that $b \mapsto G_b(x)$ is minimized by the minimizers of $\gamma^-$.

\begin{remark}
\label{rem:Schmidli}
It is worth noticing that the function $\gamma^-(b)$, $b \in [0,1]$, as defined in \eqref{gammaMeno} coincides with the opposite of the function $\beta(b)$, $b \in [0, \widetilde{b}]$, defined in \cite{EisenbergSchmidli2009} (when $\widetilde b=1$). In particular, up to a parametrization, any optimizer $b^*$ of $\gamma^-$ on $[0,1]$ does also optimize $\beta$ in \cite{EisenbergSchmidli2009}, and viceversa. We shall see in the next Theorem that, as in \cite{EisenbergSchmidli2009}, optimizers of $\gamma^-$ actually give the optimal level to be adopted. The optimal timing for reinsurance is then determined given the optimal level $b^*$ (see Section \ref{sec:opttime} below).
\end{remark}

The next result shows how to reduce the solution to \eqref{OP} to a pure optimal timing problem.
\begin{theorem}\label{verification}
Recall \eqref{OP} and \eqref{fb}. Let  $b^*\in\mathcal{B}^{*}$  and let 
$\tau^*(b^*)\in\mathcal{T}$ such that
$$ 
\tau^*(b^*) \in \textrm{argmin}_{\tau \in \mathcal T} \E\left[e^{-\rho\tau}f_{b^*}(X^{x}_{\tau^{-}})\right]
=  \textrm{argmin}_{\tau \in \mathcal T} \mathcal J(x,(\tau,b^*)),
$$
with the convention $e^{-\rho\tau}f_{b^*}(X^{x}_{\tau})=0$ on $\{\tau=\infty\}$. 
Then, the couple $a^{*}:=(\tau^*(b^*),b^{*})\in\mathcal{A}$ is an optimal reinsurance strategy (with $b^{*}$ thought of as a constant random variable).
\end{theorem}

\begin{proof}
Since $G_{b^*}(x)=\min_{b \in [0,{1}]} G_b(x)$ (see \eqref{minGb}) then
\begin{eqnarray*}
U(x)&=&G_{{1}}(x) + V(x)\\
&\ge& G_{{1}}(x) + \inf_{\tau \in \mathcal T} {\E} \left[ e^{-\rho \tau} \left(G_{b^*}(X^x_{\tau^{-}}-K)-G_{{1}}(X^x_{\tau^{-}})\right)\right]\\
&=&G_{{1}}(x)+\inf_{\tau \in \mathcal T} \mathcal{J}(x,(\tau,b^*))\\
&=&G_{{1}}(x)+\mathcal{J}(x,(\tau^*(b^*),b^*)).
\end{eqnarray*}
On the other hand
\begin{eqnarray*}
U(x)&=&G_{{1}}(x) + V(x)\\
&\le& G_{{1}}(x) + {\E} \left[ e^{-\rho \tau^*(b^*)} \left(G_{b^*}(X^x_{\tau^*(b^*)^{-}} -K)-G_{{1}}(X^x_{\tau^*(b^*)^{-}})\right)\right] \\
&=& G_{{1}}(x)+\mathcal{J}(x,(\tau^*(b^*),b^*)).
\end{eqnarray*}
Consequently $U(x) = G_{{1}}(x)+\mathcal{J}(x,(\tau^*(b^*),b^*))$ and $(\tau^*(b^*),b^*)$ is optimal.
\end{proof}

Theorem \ref{verification} provides sufficient conditions needed to identify an optimal reinsurance parameter~$b^*$. If $b^{*}\in\mathcal{B}^{*}$, then the level  corresponding to a random variable with constant value belonging to the set $\mathcal{B}^{*}$ is the second component of an optimal reinsurance strategy.  
%
%
%
%
%
%


\subsection{Optimal reinsurance timing} 
\label{sec:opttime}

Given Theorem \ref{verification}, in order to solve the optimization problem~\eqref{OP} we need to solve, for a fixed  $b^*\in\mathcal{B}^{*}$, 
the optimal stopping problem (cf.\ \eqref{fb})
\begin{eqnarray}\label{problem}
F_{b^*}(x):=\inf_{\tau\in\mathcal{T}} \mathcal{J}(x,(\tau,b^*))=
\inf_{\tau \in \mathcal T}  \E\left[e^{-\rho\tau}f_{b^*}(X^x_{\tau^{-}})\right], \ \ \ x\in[0,\infty).
\end{eqnarray}
Before addressing problem \eqref{problem} we collect properties of the obstacle function $f_{b^*}$.
 \begin{proposition}\label{propfb}
 The following hold true:
 \begin{enumerate}[(a)] 
 \item 
 $f_{b^*}$ is strictly decreasing in $[0,\hat x_{b^*}]$ and strictly increasing in $[\hat x_{b^*},\infty)$, where 
 \begin{eqnarray}\label{xHat}
 \displaystyle{\hat x_{b^*}:=\frac{\gamma^-(b^*)}{\gamma^-(b^*) - \gamma^-(1) }}K \in (K,\infty). 
 \end{eqnarray}
 \item $\lim_{x \to \infty} f_{b^*}(x)=0$.
 \item  $f_{b^*}$ is bounded. Precisely, 
we have the 
following two cases: \medskip
\begin{enumerate}[(i)]
\item If $\displaystyle{-K \gamma^-(1) \gamma^-(b^*) \le\gamma^-(b^*) - \gamma^-(1)<0 }$, then  $f_{b^*}(0)\geq 0$
(\footnote{With $f_{b^*}(0)=0$ if and only if $\displaystyle{-K= \frac{\gamma^-(b^*) - \gamma^-(1)}{\gamma^-(1) \gamma^-(b^*)}}$.}), 
and  
$$
\displaystyle{-\frac{\gamma^-(1)-\gamma^-(b^*)}{\gamma^-(1)\gamma^-(b^*)}e^{-\frac{\gamma^-(1)\gamma^-(b^*)}{\gamma^-(1)-\gamma^-(b^*)}} 
\ \le \ f_{b^*}(x) \le 
K+\frac{\gamma^-(b^*) - \gamma^-(1)}{\gamma^-(1) \gamma^-(b^*)},  \ \ \ \forall x\in[0,\infty).}
$$
\item If $\displaystyle{\gamma^-(b^*) - \gamma^-(1)}<-K \gamma^-(1) \gamma^-(b^*)$, then   $f_{b^*}(0)<0$ and  
$$
\displaystyle{-\frac{\gamma^-(1)-\gamma^-(b^*)}{\gamma^-(1)\gamma^-(b^*)}e^{-\frac{\gamma^-(1)\gamma^-(b^*)}{\gamma^-(1)-\gamma^-(b^*)}} 
\ \le \ f_{b^*}(x) < 0,  \ \ \ \forall x\in[0,\infty).}
$$
\end{enumerate}
\end{enumerate}
\end{proposition}

\begin{proof} 
See Appendix.
\end{proof}

\begin{figure}[htb!]
\begin{eqnarray*}
\begin{array}{ccc}
\includegraphics[width=0.45\textwidth]{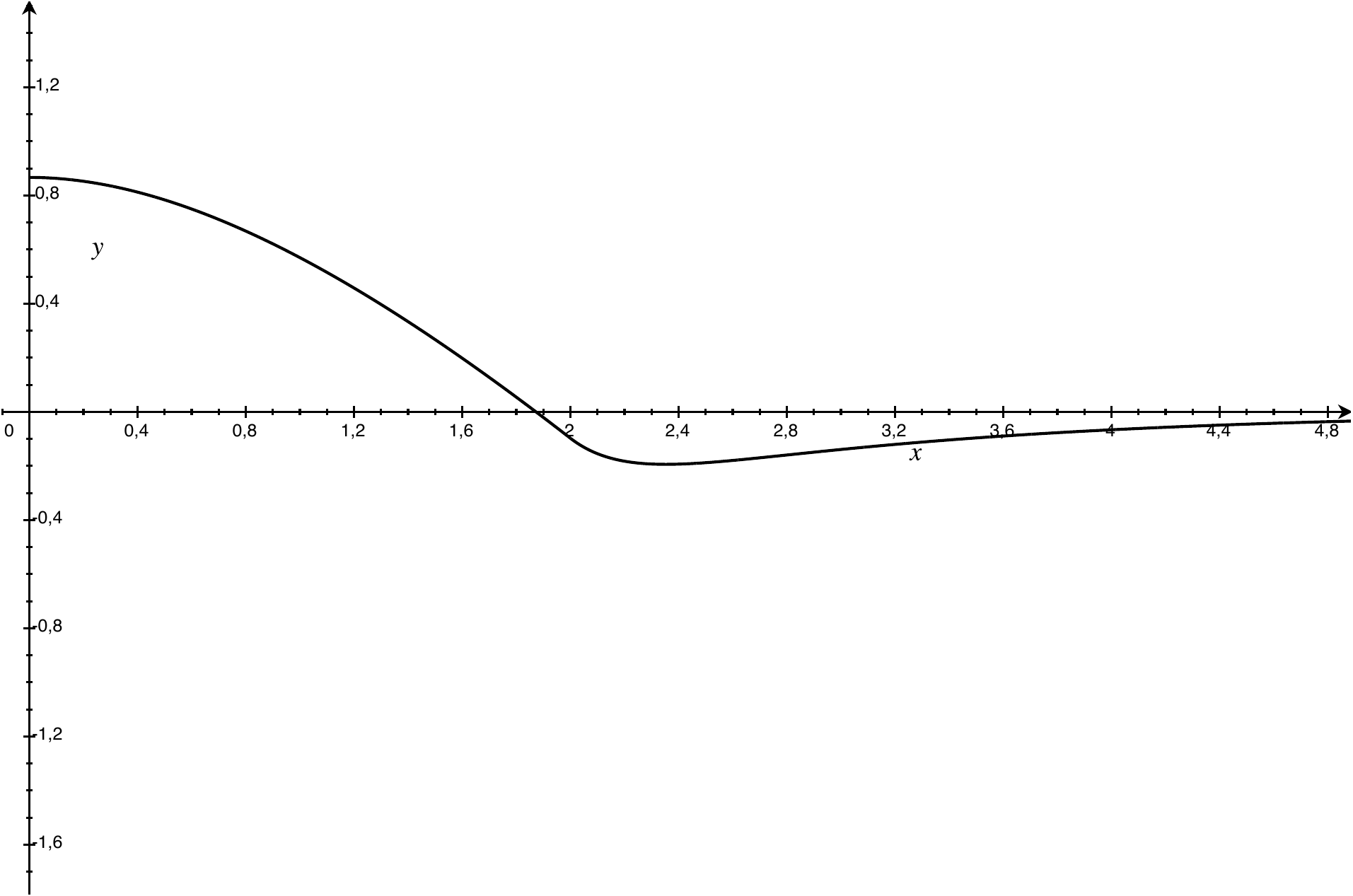} & \qquad \qquad
\includegraphics[width=0.45\textwidth]{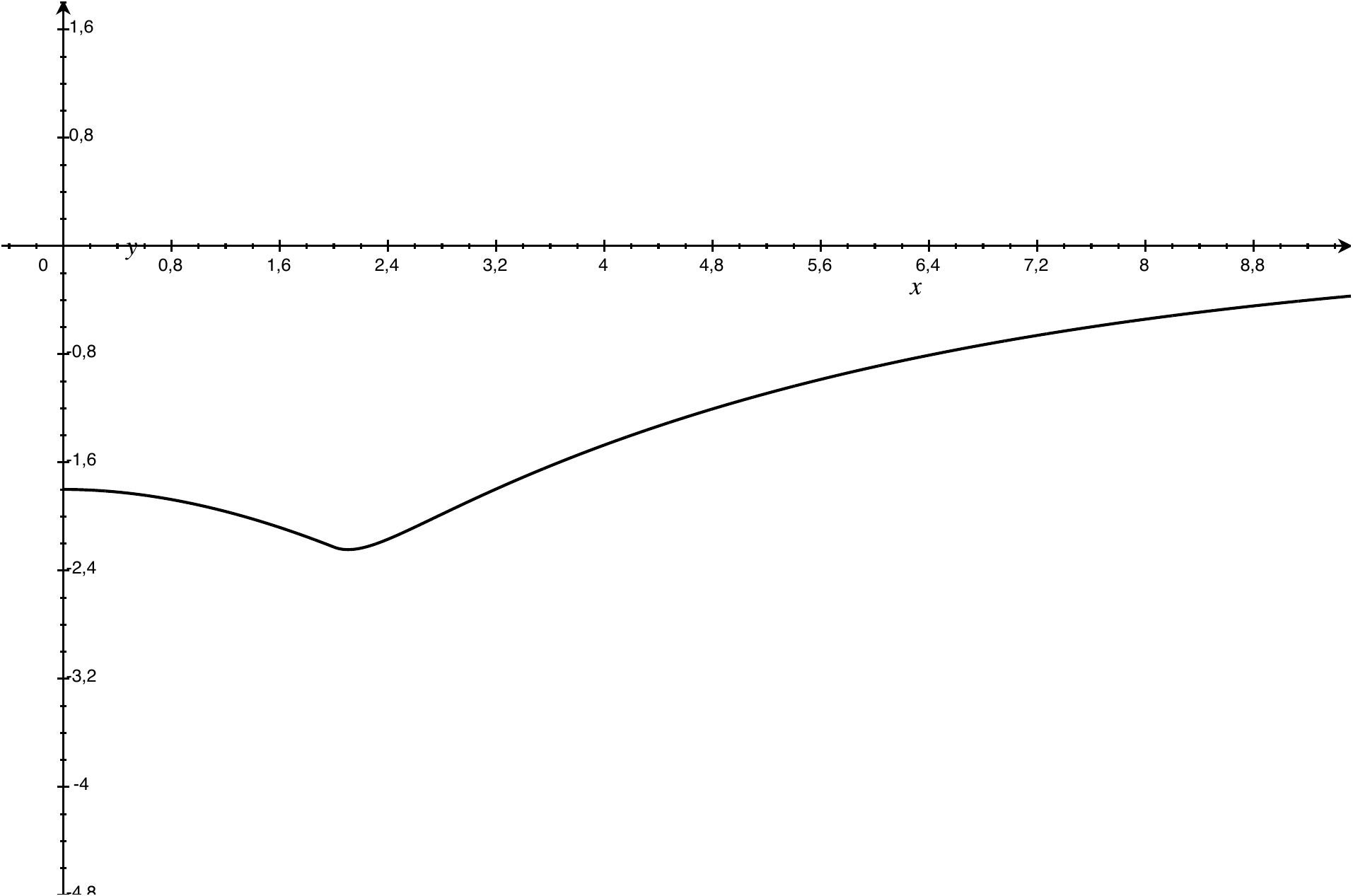}\\
\textrm{Case (i)}: \displaystyle{-K \gamma^-(1) \gamma^-(b^*) \le\gamma^-(b^*) - \gamma^-(1)<0 } &
\textrm{Case (ii)}: \displaystyle{\gamma^-(b^*) - \gamma^-(1)<-K \gamma^-(1) \gamma^-(b^*) }
\end{array}
\end{eqnarray*}
\caption{Function $f_{b^*}$ in the case $b^* \in [0,1)$.}\label{fig2}
\end{figure}

From Proposition \ref{propfb} we see that, if $1 \in \mathcal B^*$, then $0 \le f_1(x) \le K$, $\forall x\ge 0$. This in turn leads to the fact that it it is never optimal to start the reinsurance contract and $F_1 \equiv 0$ (cf.\ \eqref{problem}). Hence, in the rest of the section, we assume that $1 \not\in \mathcal B^*$.

Problem \eqref{problem} is a one-dimensional optimal stopping problem for a reflected diffusion that can be addressed by techniques of the 
Dynamic Programming Principle. To that end, set
\begin{eqnarray}\label{LX1}
\mathcal L\varphi :=\frac{1}{2}\lambda \sigma^2 \varphi''
+\lambda\eta\mu\varphi', \ \ \ \varphi\in C^{2}((0,\infty);\R). 
\end{eqnarray}
By classical dynamic programming arguments, the optimization problem \eqref{problem} is expected to be associated to the variational problem:
\begin{equation}\label{VI}
\begin{cases}
\min\left\{(\mathcal{L}-\rho)w(x), \ f_{b^*}(x)-w(x)\right\}=0, \ \ \ x\in(0,\infty),\\
w'(0)=0.
\end{cases}
\end{equation}

As a matter of fact, one has the following verification theorem. 
\begin{theorem}\label{th:ver}
Let $w\in W^{2,\infty}_\textrm{loc}([0,\infty);\R)$ be a bounded solution to \eqref{VI} and define the reinsurance region
$$\mathcal{R}:=\{x\in[0,\infty): \ w(x)\ge f_{b^*}(x)\}.$$
 Then $w=F_{b^*}$ and the entry time 
$$\tau^{*}(b^*):=\inf\{t\geq 0: \ X_t^{x}\in \mathcal{R}\}$$
(with the convention $\inf\emptyset=\infty$)
is the optimal reinsurance time for \eqref{OP}; that is, one has $F_{b^*}(x)=\mathcal{J}(x,(\tau^{*}(b^*), b^*))$.
\end{theorem}

\begin{proof} 
See Appendix.
\end{proof}

Now, given $b^*\in\mathcal{B}^{*}$, with regards to the properties of $f_{b^*}$ collected in Proposition \ref{propfb}, we expect that 
$\mathcal R= \{x \ge 0 \ : \ x \ge x^*_{b^*}\}$ for some $x^*_{b^*}>0$ and that \eqref{problem} is thus related to 
the following free-boundary problem: Find $(\hat w,x^*_{b^*}) \in C^2((0,x^*_{b^*});\R) \times  \R_+$ such that:
\begin{eqnarray}\label{FB}
\begin{cases}
(\mathcal{L}-\rho)\hat w(x)=0,  \mbox{if} \ x\in (0,x^*_{b^*}),\\\\
\hat w'(0)=0,\\\\
\hat w(x^{*}_{b^*}) = f_{b^*}(x^{*}_{b^*}), \ \  \hat w'(x^{*}_{b^*}) = f'_{b^*}(x^{*}_{b^*}).
\end{cases}
\end{eqnarray}
Note  that the general solution to $(\mathcal{L}-\rho)w=0$ is 
\begin{eqnarray}\label{functionH}
h(x;C_1,C_2)=C_1e^{\gamma^-(1)x}+C_2e^{\gamma^+(1)x}, \quad C_1,C_2 \in \mathbb R,
\end{eqnarray}
where (see \eqref{pol}) $\gamma^-(1), \gamma^+(1)$ are, respectively,  the negative and the positive  solutions to
\begin{eqnarray}\label{pol1}
\Phi(1,\gamma)=\frac{1}{2}\lambda \sigma^2 \gamma^2+
\lambda \eta\mu\gamma-\rho, \quad \gamma \in  \mathbb R.
\end{eqnarray}

The following result characterizes the solution to \eqref{FB}.
\begin{proposition}\label{prop1}
\begin{enumerate}[(i)]
\item[]
\item The equation 
\begin{eqnarray}\label{eqToSolve}
&&\gamma^+(1)(\gamma^-(b^*)-\gamma^-(1))e^{(\gamma^-(b^*)-\gamma^+(1))x}-\gamma^-(1)(\gamma^-(b^*)-\gamma^+(1))e^{(\gamma^-(b^*)-\gamma^-(1))x} \nonumber\\
&&=\gamma^-(b^*)(\gamma^+(1)-\gamma^-(1))e^{\gamma^-(b^*)K}
\end{eqnarray}
admits a unique solution $x^{*}_{b^*}$ in the interval $[0,\infty)$.
\item The function $K  \mapsto x^{*}_{b^*}= x^{*}_{b^*}(K)$ is strictly increasing  and 
$$K\  \le \ x^{*}_{b^*} \ \le \ \hat x_{b^*}:=\displaystyle{\frac{\gamma^-(b^*)}{\gamma^-(b^*) - \gamma^-(1) }}K.$$ 
\item 
The unique solution to \eqref{FB} is 
\begin{eqnarray}\label{wSol}
\hat w(x)=
h(x;C_1,C_2), \quad x\in [0,x^{*}_{b^*}), 
\end{eqnarray}
where
\begin{eqnarray}\label{expressionC1C2_caseI}
C_1:=\displaystyle{\frac{1}{\gamma^-(1)} B(x^*_{b^*})}, \ \ \ \ \ \ \ \ C_2:=\displaystyle{-\frac{1}{\gamma^+(1)}B(x^*_{b^*})},
\end{eqnarray} 
\begin{eqnarray}\label{expressionB}
B(x^*_{b^*}):=\displaystyle{ \gamma^+(1)(\gamma^-(b)-\gamma^-(1))\frac{e^{\gamma^-(1) x^*_{b^*}}}{H(x^*_{b^*})}   } \in (0,1).
\end{eqnarray}
with
\begin{eqnarray}\label{expressionH}
H(x^*_{b^*}):=\gamma^+(1)\left(\gamma^-(b^*)-\gamma^-(1)\right)e^{\gamma^-(1)x^*_{b^*}} - \gamma^-(1)\left(\gamma^-(b^*)-\gamma^+(1)\right)e^{\gamma^+(1)x^*_{b^*}}<0.
\end{eqnarray}
\end{enumerate}
\end{proposition}

The next result provides the needed link between \eqref{VI} and \eqref{FB}.
\begin{theorem}\label{prop2}
Let $(\hat w, x^*_{b^*})$ be the solution to \eqref{FB} provided in Proposition \ref{prop1}. Then, 
$$w(x):= \begin{cases}
\hat w(x), \ \ \ \ \ \  \mbox{if} \ x\in[0,x^{*}_{b^{*}}),\\
f_{b^{*}}(x) \ \ \ \ \ \ \mbox{if} \ x\ge x_{b^{*}}^{*},
\end{cases}
$$ 
is a bounded solution to \eqref{VI} such that $w \in W^{2,\infty}_{loc}([0,\infty);\R)$. Therefore Theorem \ref{th:ver} applies: $w=F_{b^*}$ on $[0,\infty)$ and $\tau^*(b^*)=\inf\{t \ge 0 \ : \ X_t^x \ge x^*_{b^*}\}$ is optimal for \eqref{problem}.
\end{theorem}

\begin{proof} 
See Appendix.
\end{proof}

Theorem \ref{prop2} provides the optimal reinsurance rule. In particular, once the optimal level $b^* \in \mathcal{B}^*$ has been determined by minimizing the function $\gamma^-$ over $[0,1]$, it is optimal to reinsure at the first time at which the company's surplus process exceeds the critical level $x^*_{b^*}$. Such a trigger value is completely determined as the unique solution to \eqref{eqToSolve}. Since the level $b^*$ is independent of $K$, due to Proposition \ref{prop1}-(ii), we see that, when the fixed cost $K\downarrow 0$, also $x^*_{b^*} \downarrow 0$, so that it is optimal to immediately undertake reinsurance. Given that our $b^*$ coincides, up to a parametrization, with that of \cite{EisenbergSchmidli2009} (cf.\ Remark \ref{rem:Schmidli}), we can thus observe that our optimal policy is consistent with that of \cite{EisenbergSchmidli2009} in the limit of vanishing fixed cost.

\section{Examples and Illustrations}
\label{sec:applications}

In this section we apply the general findings to the two relevant cases of proportional and excess-of-loss reinsurance.

\subsection{Proportional reinsurance}
\label{proportional}

We consider the case of proportional reinsurance. The retention level of the insurer is given by the function $r(z, b) = bz$, 
for each $b \in [0,1]$ (cf.\ \eqref{propReins}); consequently, $M_1(b)= \mu b$ and $M_2(b) = \sigma^2 b^2$. 
The surplus process $X_t^x$ (cf.\ \eqref{surplus}) then becomes 
\begin{eqnarray}
\label{X-ex-1}
X^{x}_{t}= X^{x}_{\tau^{-}}-K+ \lambda \mu(\theta b_{\tau} - (\theta - \eta))(t-\tau)+ \sqrt{\lambda} \sigma b_\tau(W_{t}-W_{\tau}) + I_{t}, \ \ \ t\geq \tau.
\end{eqnarray}

In this case, equation \eqref{pol} reads:
\begin{eqnarray}\label{polProp}
\Phi(b,\gamma):=\frac{1}{2}\lambda \sigma^2 b^2 \gamma^2+
\lambda \mu \left(\theta b- (\theta - \eta)\right)\gamma-\rho, \quad \gamma \in  \mathbb R.
\end{eqnarray}
The next result characterizes $\mathcal B^*$.
\begin{proposition}\label{propInternalMin}
It holds
\begin{itemize}
\item[(i)] $\mathcal B^*=\{b^*\}$, with $b^* \in (0,1]$;
\item[(ii)] We have 
\begin{eqnarray}\label{cond:proportional}
b^*<1 \quad\iff\quad \eta < \theta < \eta + \displaystyle{\sqrt{\eta^2+\frac{2\rho \sigma^2}{\lambda \mu^2}}}.
\end{eqnarray}
In this case, $b^*$ is the unique solution to the equation 
\begin{eqnarray}\label{equationPhi}
\phi(b)=0
\end{eqnarray}
where
\begin{eqnarray}\label{phi}
\phi(b) = \sigma^2 b\gamma^-(b)+\mu \theta.
\end{eqnarray}
\end{itemize}
\end{proposition}

\begin{proof}
See Appendix.
\end{proof}

We next illustrate numerically the sensitivity of the optimal level $b^*$ and of the optimal reinsurance boundary $x^*_{b^*}$ with respect some relevant model's parameters. We choose benchmark values of the parameters as it follows: $\theta=0.5$, $\eta=0.3$, $\lambda=0.05$, $\rho=0.04$, $\mu=10$, $\sigma^2=200$, $K=10$.
With such a choice of values, condition \eqref{cond:proportional} is satisfied, and from \eqref{equationPhi} (see Proposition \ref{propInternalMin}) and
Proposition \ref{prop1} we compute $b^*=0.0580$ and $x^*_{b^*}=12.2341$.

Figure \ref{fig1} shows how $b^*$ varies with the parameters $\rho$, $\mu$ and $\sigma^2$.
We observe that the level $b^*$ is decreasing with respect to the time-preference factor $\rho$ of the insurer: The more the insurer is impatient, the less is the (discounted) cost paid at time $\tau$, the more is convenient to reinsure in order to minimize the amplitude of the Brownian fluctuations and therefore the probability of capital injections being necessary. Furthermore, if $\mu$ increases, the drift of the surplus process decreases (see \eqref{X-ex-1}), thus increasing the probability of the need of additional capital injections. Hence, the company reinsures less, i.e.\ $b^*$ increases, in order to mitigate such an effect.
Finally, we see that if $\sigma^2$ increases, hence if the size of the Brownian risk increases, the insurance company passes more risk on. 

\begin{figure}[htb!]
\begin{eqnarray*}
\begin{array}{ccccccc}
\includegraphics[width=0.2\textwidth]{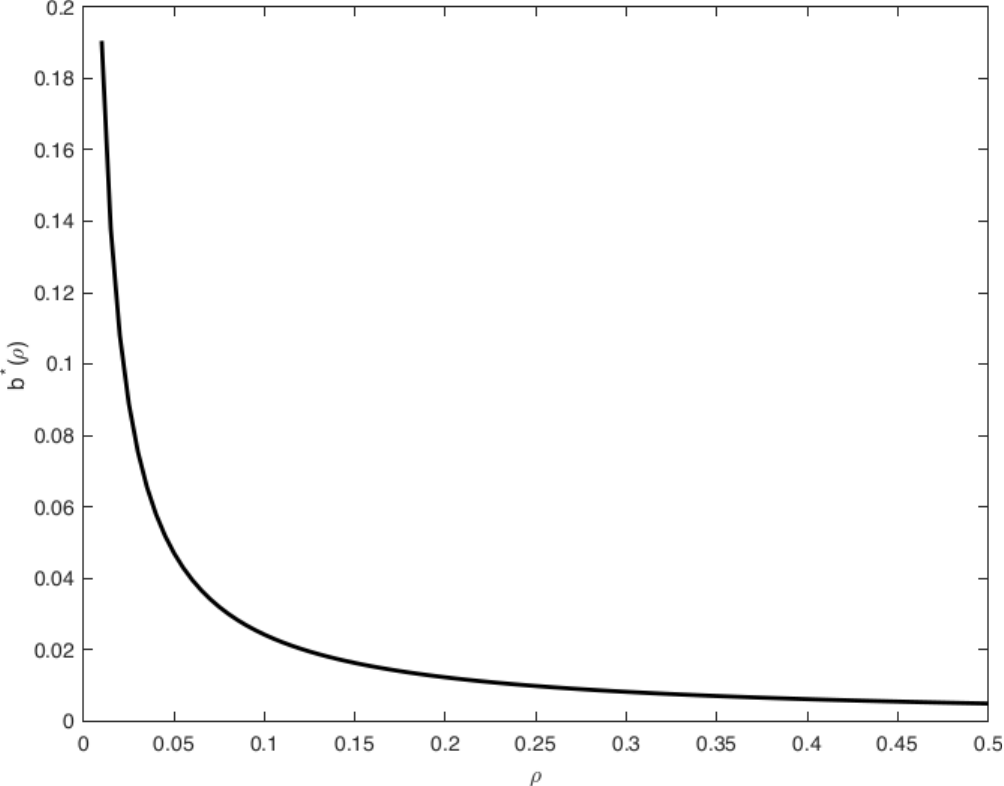} &&
\includegraphics[width=0.2\textwidth]{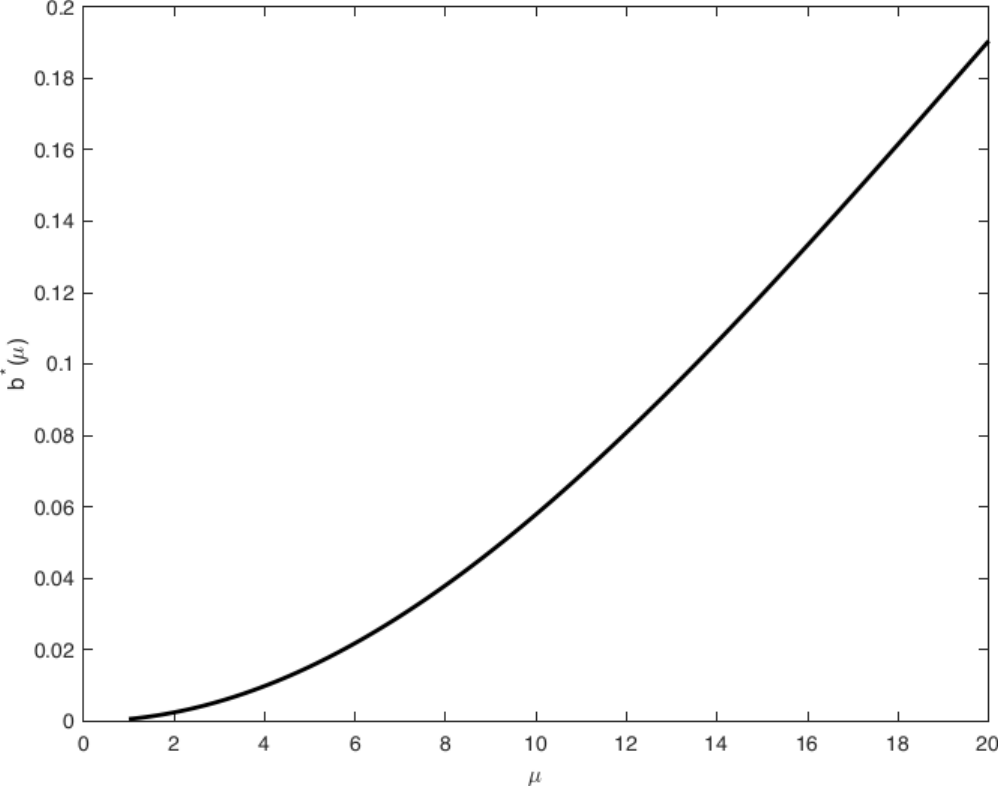} &&
\includegraphics[width=0.2\textwidth]{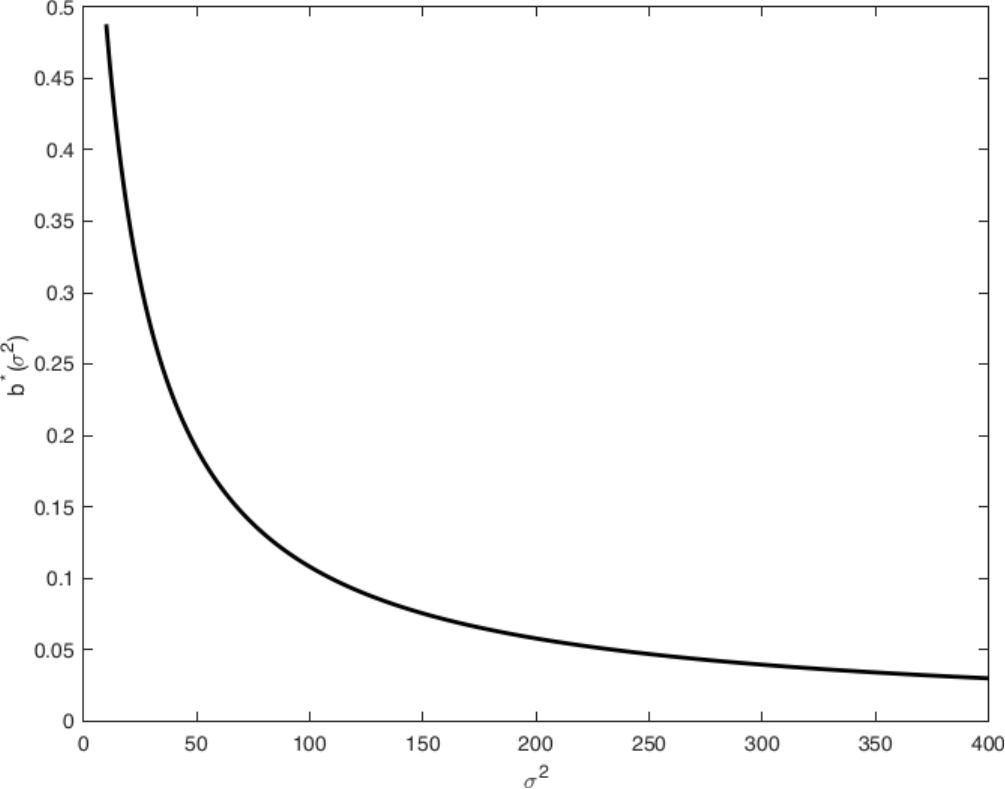} &&
\phantom{\includegraphics[width=0.2\textwidth]{proportional_bStarSigma}}\\
b^*(\rho)  && b^*(\mu) && b^*(\sigma^2)\\ 
\end{array}
\end{eqnarray*}
\caption{Dependency of $b^*$ with respect to $\rho$, $\mu$ and $\sigma^2$.}
\label{fig1}
\end{figure}

Figure \ref{fig2} shows how $x^*_{b^*}$ depends on the parameters $\rho$, $\mu$, $\sigma^2$ and $K$. 
Our numerical example reveals that, if $\rho$ increases, then the insurance company becomes more impatient and anticipates reinsurance. Increasing monotonicity of $x^*_{b^*}$ is instead observed with respect to the transaction cost $K$ and the parameter $\mu$. The larger $K$ is, the more expensive is the reinsurance contract, and the later is its starting time. Also, if $\mu$ increases, the trend of the surplus process decreases (see again \eqref{X-ex-1}), thus inducing the company to postpone reinsurance and hence to keep $b=1$ for a longer time period in order to reduce the negative growth of $X$ (in absolute value) and consequently the possibility of capital injections. Finally, if $\sigma^2$ increases, the amplitude of the Brownian fluctuations becomes more relevant, this calling for an earlier reinsurance aiming at mitigating the increase of the probability of additional capital injections. 

\begin{figure}[htb!]
\begin{eqnarray*}
\begin{array}{ccccccc}
\includegraphics[width=0.2\textwidth]{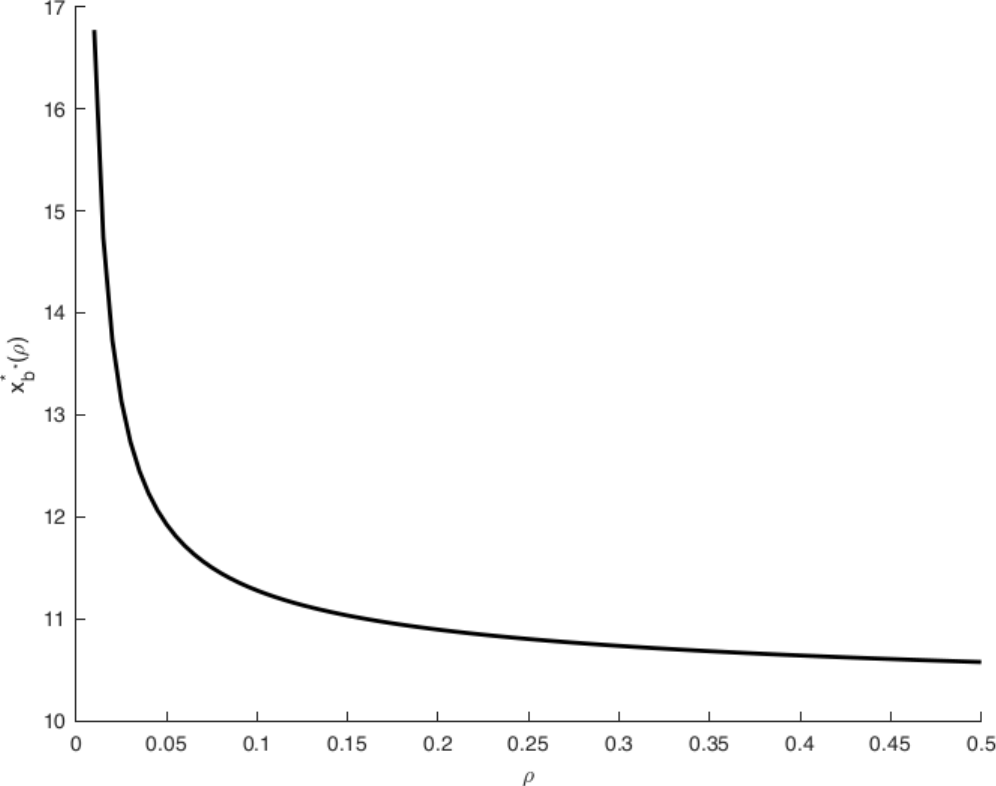} &&
\includegraphics[width=0.2\textwidth]{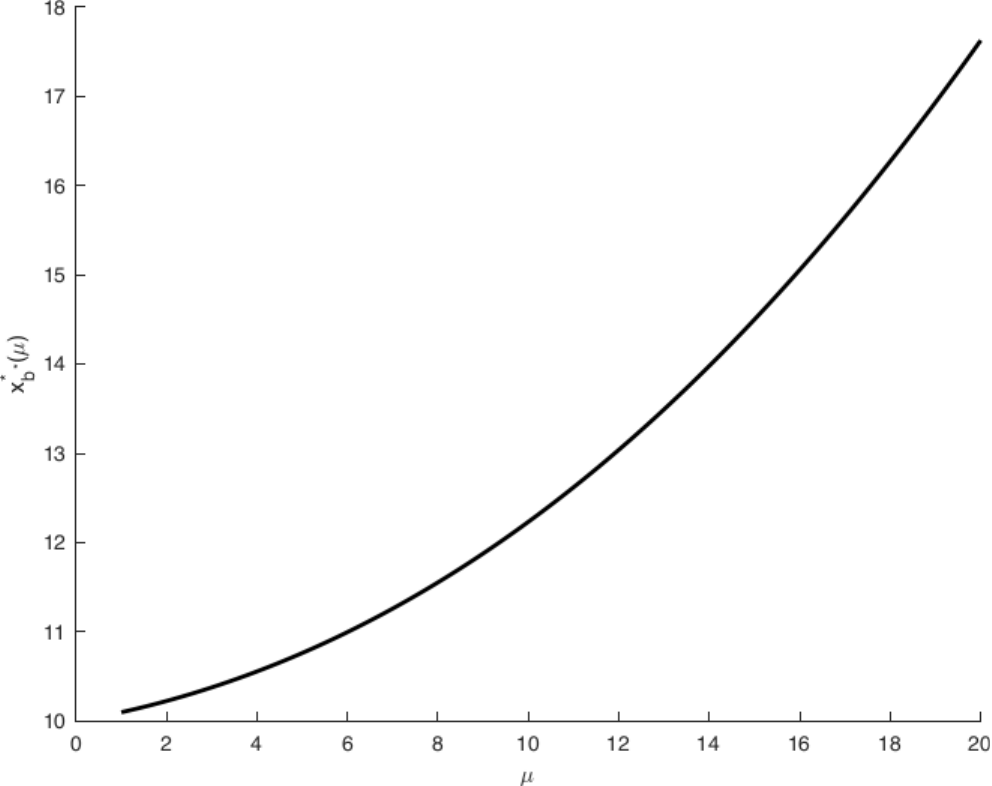} &&
\includegraphics[width=0.2\textwidth]{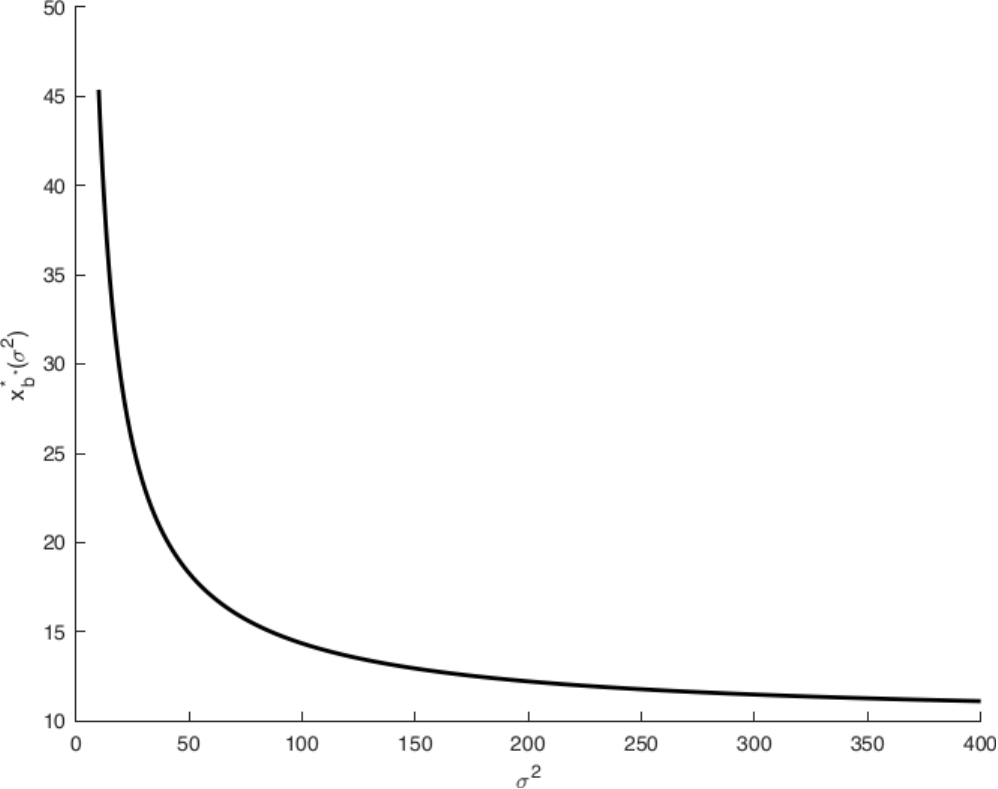} &&
\includegraphics[width=0.2\textwidth]{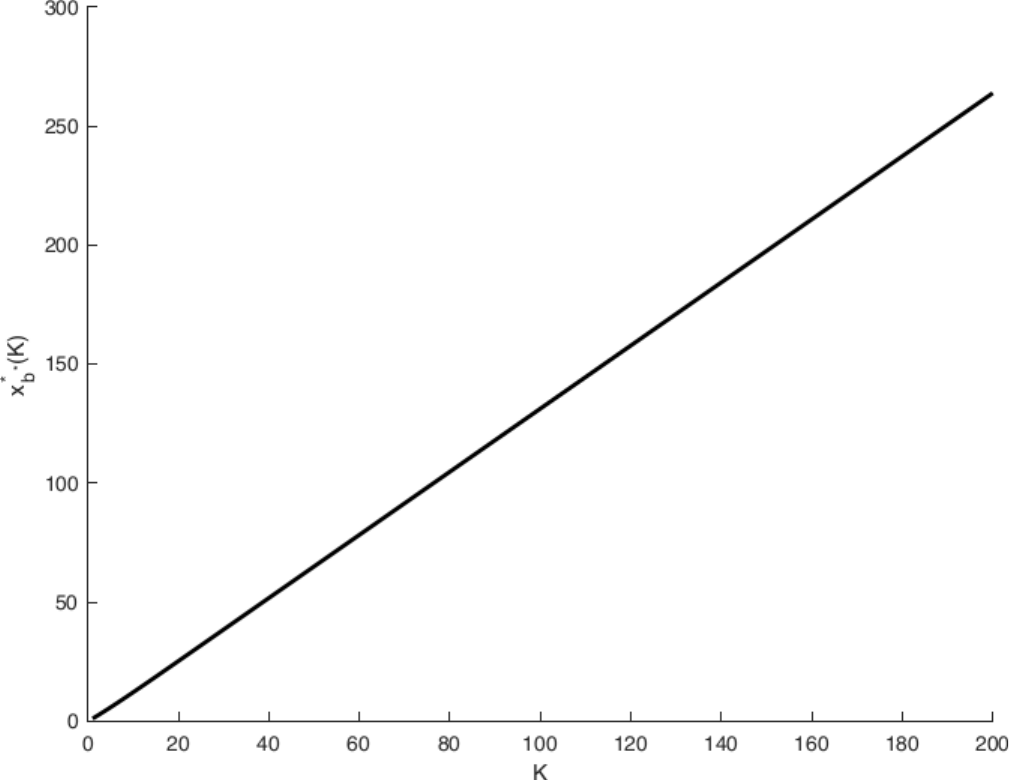} \\
x^*_{b^*}(\rho) && x^*_{b^*}(\mu) && x^*_{b^*}(\sigma^2) && x^*_{b^*}(K)
\end{array}
\end{eqnarray*}
\caption{Dependency of $x^*_{b^*}$ with respect to $\rho$, $\mu$, $\sigma^2$ and $K$.}
\label{fig2}
\end{figure}


\subsection{Excess-of-loss reinsurance}
\label{excessOfLoss}

We here consider the case of excess-of-loss reinsurance. The retention level of the insurer is given by the function 
$r(z, b) = z \wedge \frac{b}{1-b}$, for each $b \in [0,1]$ (cf.\ \eqref{excessOfLossReins}). We assume that the distribution of the claim sizes has a density $p$ with respect to the Lebesgue measure; that is, $\nu_Z(dz)=p(z) dz$. 
The next result characterizes $\mathcal B^*$.
\begin{proposition}\label{propEOL}
It holds $b^* \in \mathcal B^*$ if and only if $b^* \in (0,1)$ is a solution to the equation 
\begin{eqnarray}\label{equationPsi}
\psi(b)=0, 
\end{eqnarray}
where
\begin{eqnarray}\label{psi}
\psi(b) := \frac{b}{1-b}\gamma^-(b)+ \theta, \quad b\in[0,1].
\end{eqnarray}
\end{proposition}
\begin{proof}
See Appendix.
\end{proof}

\subsubsection{The case $Z \sim \textrm{Exp}(1/\mu)$}
\label{section:EoLExp}

We assume here that the claim sizes are exponentially distributed with $\nu_Z(dz)=\frac{1}{\mu}e^{-z/\mu} dz$. Consequently, by some easy computations, we find
\begin{eqnarray*}
M_1(b)&=&\mu\left(1-e^{-\frac{b}{(1-b)\mu}}\right)\\
M_2(b) &=& 2\mu\left(\mu-\left(\frac{b}{1-b}+\mu\right) e^{-\frac{b}{(1-b)\mu}} \right).
\end{eqnarray*}
The surplus process $X_t^x$ (cf.\ \eqref{surplus}) thus becomes 
\begin{eqnarray*}
X^{x}_{t}&=& X^{x}_{\tau^{-}}-K+ \lambda \mu \left(\eta -\theta e^{-\frac{b_\tau}{(1-b_\tau)\mu}} \right)(t-\tau)+\\
&&+ \sqrt{2\lambda \mu\left(\mu-\left(\frac{b_\tau}{1-b_\tau}+\mu\right) e^{-\frac{b_\tau}{(1-b_\tau)\mu}}\right)}(W_{t}-W_{\tau}) 
+ I_{t}, \ \ \ t\geq \tau,
\end{eqnarray*}
while equation \eqref{pol} now reads
\begin{eqnarray*}
\Phi(b,\gamma):=\lambda\mu\left(\mu-\left(\frac{b}{1-b}+\mu\right) e^{-\frac{b}{(1-b)\mu}}\right) \gamma^2+
\lambda \mu \left(\eta -\theta e^{-\frac{b}{(1-b)\mu}}\right)\gamma-\rho, \quad \gamma \in  \mathbb R.
\end{eqnarray*}

We illustrate numerically the sensitivity of the optimal level $b^*$ and of the optimal reinsurance boundary $x^*_{b^*}$ with respect some relevant model's parameters. We choose benchmark values of the parameters as it follows: $\theta=0.5$, $\eta=0.3$, $\lambda=0.05$, $\rho=0.04$, $\mu=10$, $K=10$.
With such a choice, by Proposition \ref{prop1} and from \eqref{equationPsi} (see Proposition \ref{propEOL}), we obtain $b^*=0.4627$ and $x^*_{b^*}=11.4785$.

Figure \ref{fig3} shows how $b^*$ depends on the parameters $\rho$ and $\mu$. With respect to $\rho$, we observe a behavior of $b^*$ which is line with that we had in the case of proportional reinsurance, which can be then explained in the same way. Noticing that under the Exponential distribution of the claim's size the parameter $\mu$ measures both the amplitude of the Brownian fluctuations and the trend of the average profits of the company, we observe that an increase in $\mu$ leads the company to reinsure less, i.e.\ to an increase of $b^*$. Hence, in the interplay of the two roles played by $\mu$, the drift effect appears to be dominant and the behavior of $\mu \mapsto b^*(\mu)$ can be thus explained as in the case of the proportional reinsurance (cf.\ Figure \ref{fig1}).

\begin{figure}[htb!]
\begin{eqnarray*}
\begin{array}{ccccc}
\includegraphics[width=0.27\textwidth]{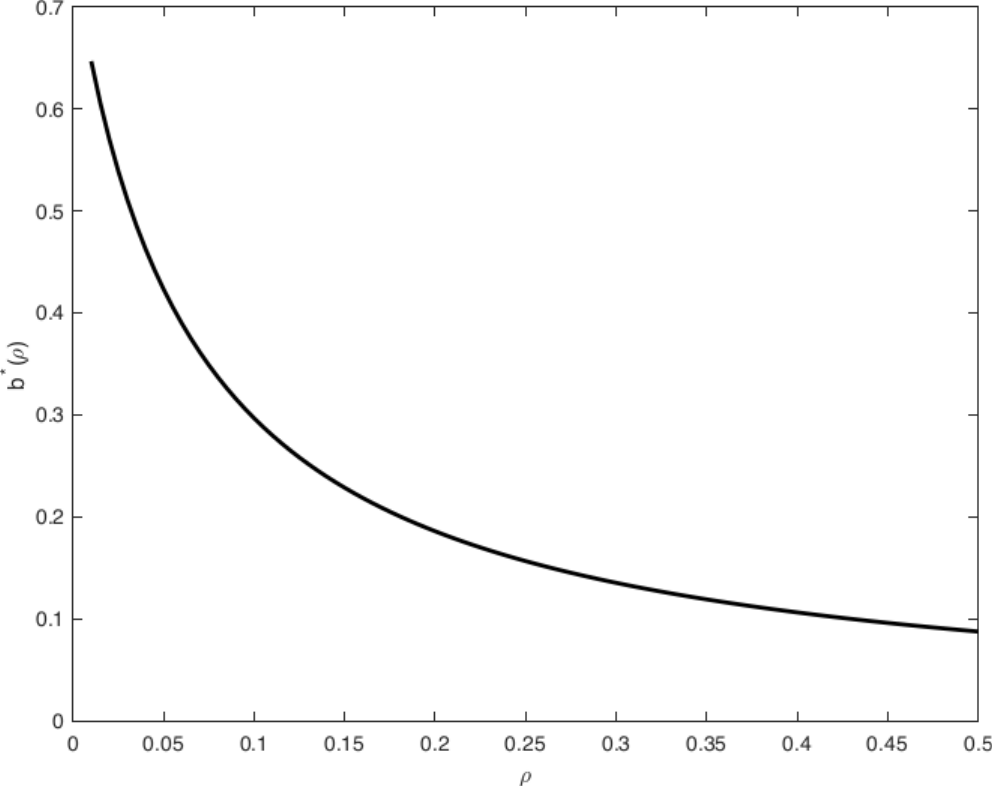} &&
\includegraphics[width=0.27\textwidth]{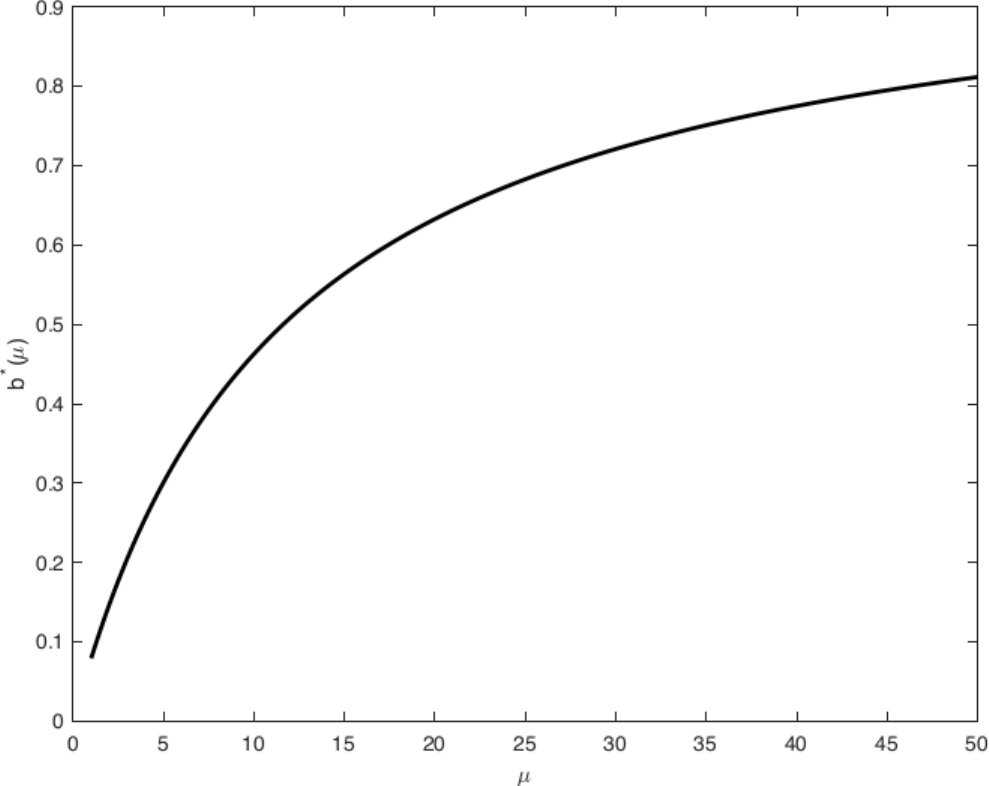} &&
\phantom{\includegraphics[width=0.2\textwidth]{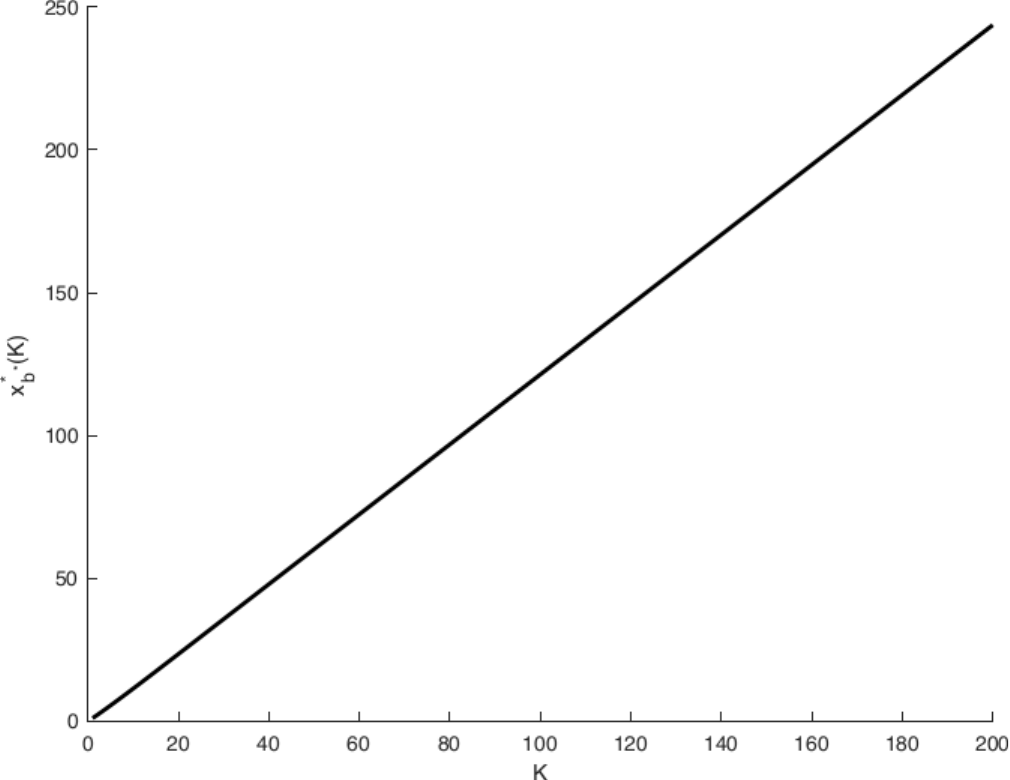}} \\
b^*(\rho) &&  b^*(\mu) && \\ 
\end{array}
\end{eqnarray*}
\caption{Dependency of $b^*$ with respect to $\rho$ and $\mu$.}\label{fig3}
\end{figure}

Figure \ref{fig4} plots $x^*_{b^*}$ as a function of the parameters $\rho$, $\mu$ and $K$. 
While the behavior of $x^*_{b^*}$ with respect to $\rho$ and $K$ can be explained by the same reasoning that we followed in the case of proportional reinsurance, a different pattern is observed for the dependency of $x^*_{b^*}$ with respect to $\mu$. Here, we see that $x^*_{b^*}$ is decreasing with respect to $\mu$. Recalling again that $\mu$ measures both the amplitude of the Brownian fluctuations and the trend of the average profits of the company, we find that in the interplay of the two roles played by $\mu$, the volatility effect appears to be dominant and the behavior of $\mu \mapsto b^*(\mu)$ can be thus explained as that of $\sigma^2 \mapsto b^*(\sigma^2)$ in the case of the proportional reinsurance (cf.\ Figure \ref{fig2}).

\begin{figure}[htb!]
\begin{eqnarray*}
\begin{array}{ccccccc}
\includegraphics[width=0.27\textwidth]{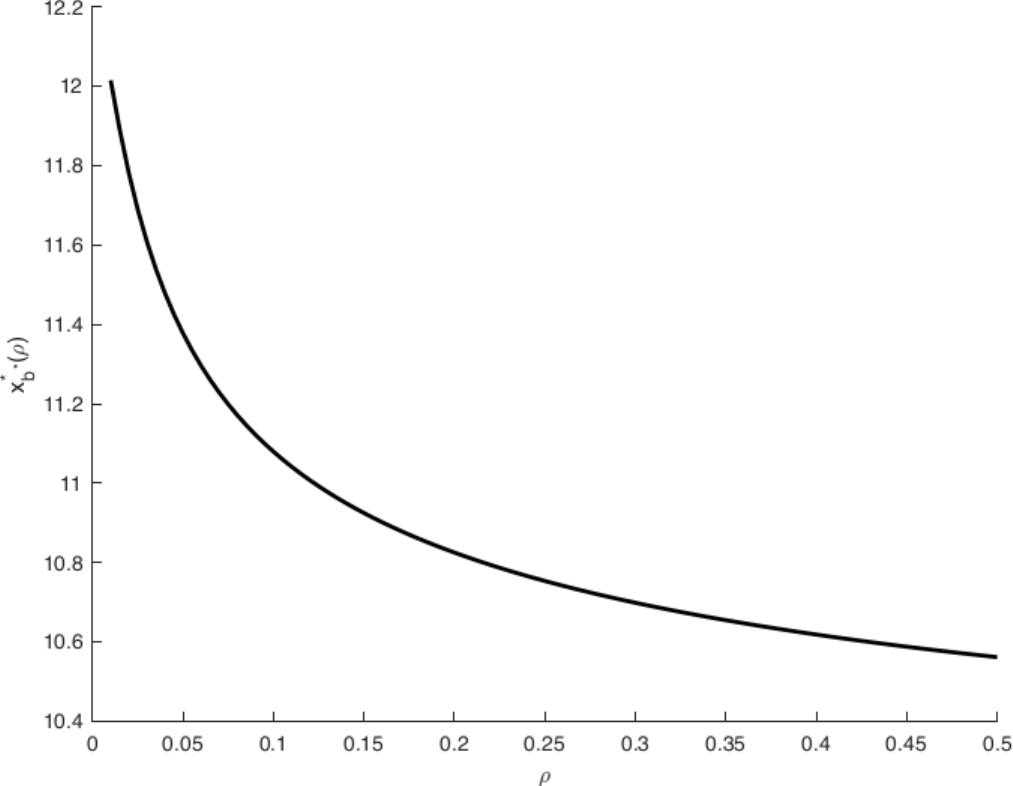} &&
\includegraphics[width=0.27\textwidth]{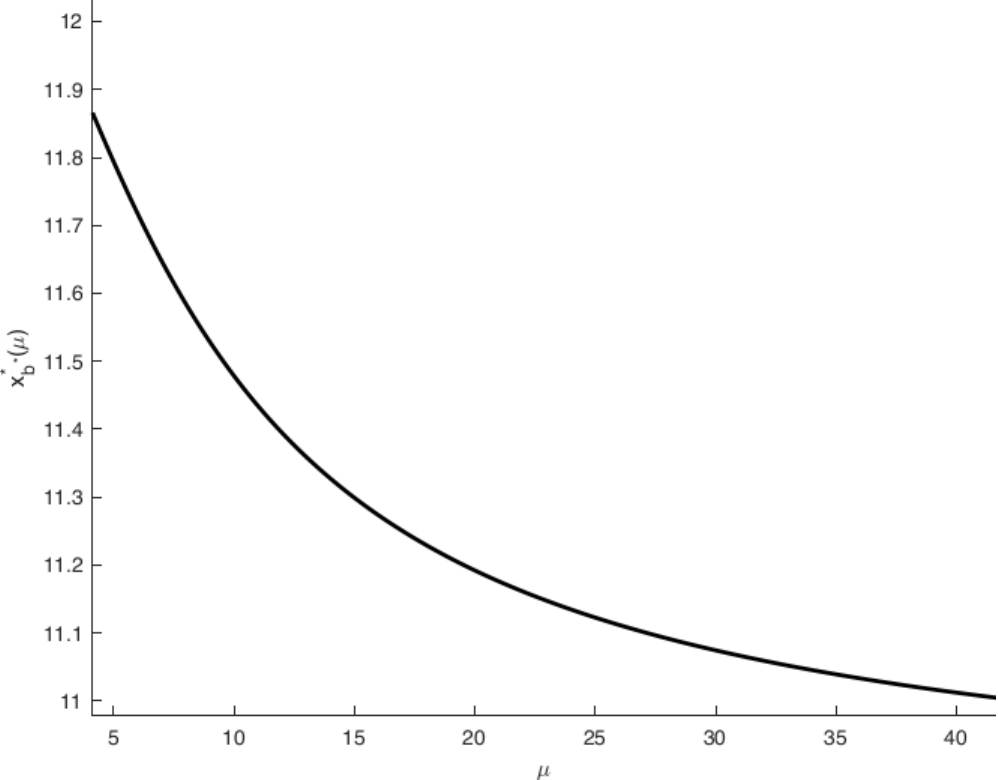} &&
\includegraphics[width=0.27\textwidth]{EOL_exponential_xStarK}\\
x^*_{b^*}(\rho) &&  x^*_{b^*}(\mu) && x^*_{b^*}(K)\\ 
\end{array}
\end{eqnarray*}
\caption{Dependency of $x^*_{b^*}$ with respect to $\rho$, $\mu$, and $K$.}\label{fig4}
\end{figure}

\subsubsection{The case $Z \sim \textrm{Pareto}(\zeta,\alpha)$}

We assume that the claim sizes $Z \sim \textrm{Pareto}(\zeta, \alpha)$, $\alpha > 2$, with density
$$
p(z) = \begin{cases}
0 & \textrm{if } z < \zeta\\
\displaystyle{\frac{\alpha \zeta^\alpha}{z^{\alpha+1}}}  & \textrm{if } z \ge \zeta.
\end{cases}
$$ 
By some computations, we find
\begin{eqnarray*}
M_1(b)&=&\begin{cases}
\displaystyle{\frac{b}{1-b}} & \textrm{if } \displaystyle{b < \frac{\zeta}{1+\zeta}}\\
\displaystyle{\frac{\zeta}{\alpha-1}\left(\alpha-\left(\zeta \frac{1-b}{b}\right)^{\alpha-1}\right)} & \textrm{if } \displaystyle{b \ge \frac{\zeta}{1+\zeta}}
\end{cases}\\
M_2(b) &=& \begin{cases}
\displaystyle{\left(\frac{b}{1-b}\right)^2} & \textrm{if } \displaystyle{b < \frac{\zeta}{1+\zeta}}\\
\displaystyle{\frac{\zeta^2}{\alpha-2}\left(\alpha-2 \left(\zeta \frac{1-b}{b}\right)^{\alpha-2}\right)} & \textrm{if } \displaystyle{b \ge \frac{\zeta}{1+\zeta}}.
\end{cases}
\end{eqnarray*}
Equation \eqref{pol} now reads
\begin{eqnarray*}
\Phi(b,\gamma):= \begin{cases}
\frac{1}{2}\lambda \left(\frac{b}{1-b}\right)^2 \gamma^2+
\lambda \left[\theta \left(\frac{b}{1-b} - \zeta\right) +\eta\zeta\right]\gamma-\rho
& \textrm{if } b < \frac{\zeta}{1+\zeta}\\
\frac{1}{2}\lambda  \frac{\zeta^2}{\alpha-2}\left(\alpha-2 \left(\zeta \frac{1-b}{b}\right)^{\alpha-2}\right) \gamma^2+
\lambda  \left[\frac{\theta\zeta}{\alpha-1} \left( 1-\left(\zeta \frac{1-b}{b}\right)^{\alpha-1}\right)+\eta \zeta\right]\gamma-\rho 
& \textrm{if } b \ge \frac{\zeta}{1+\zeta},
\end{cases}
\end{eqnarray*}
with $\gamma \in  \mathbb R$.
%

We illustrate numerically the sensitivity of the optimal level $b^*$ and of the optimal reinsurance boundary $x^*_{b^*}$ with respect some relevant model's parameters. We choose benchmark values of the parameters as it follows: $\theta=0.5$, $\eta=0.3$, $\lambda=0.05$, $\rho=0.04$, $\zeta=10$, $K=10$.
With such a choice, by Proposition \ref{prop1} and from \eqref{equationPsi} (see Proposition \ref{propEOL}), we obtain $b^*=0.5195$ and $x^*_{b^*}=11.7572$.

Figure \ref{fig5} shows how $b^*$ depends on the parameters $\rho$ and $\zeta$, while Figure \ref{fig6} plots $x^*_{b^*}$ as a function of the parameters $\rho$, $\zeta$ and $K$. We observe behaviors of $b^*$ and $x^*_{b^*}$ which are similar to those observed with respect to $\rho$, $\mu$ and $K$ in the case of an Exponential distribution of the claim size, and which can be then explained through the same rationale. As a matter of fact, in the Pareto distribution, the average and the variance of the sizes of the claims are increasing functions of the only parameter $\zeta$, just as they are functions of $\mu$ in the case of an Exponential distribution.

\begin{figure}[htb!]
\begin{eqnarray*}
\begin{array}{ccccc}
\includegraphics[width=0.27\textwidth]{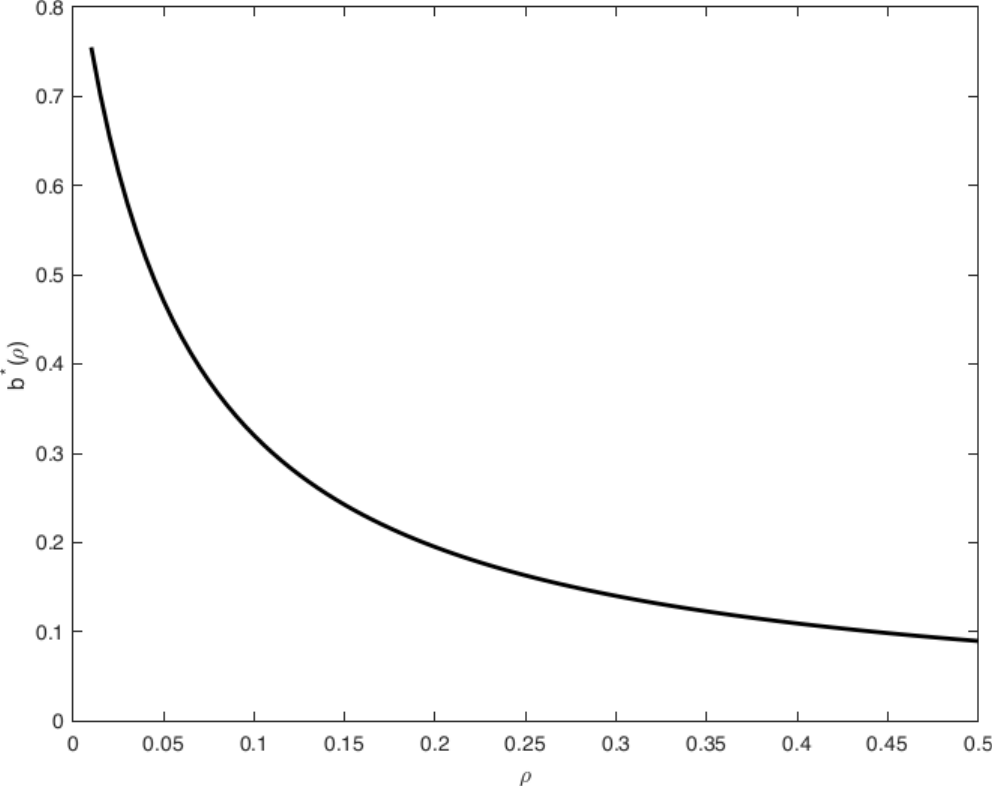} &&
\includegraphics[width=0.27\textwidth]{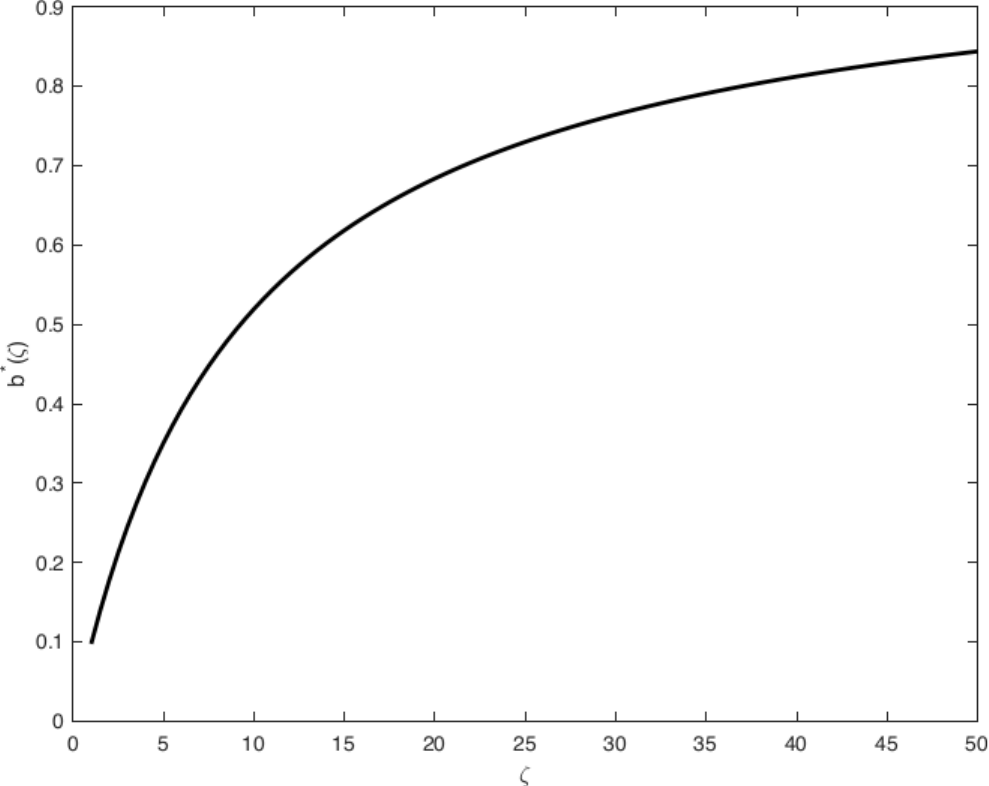} &&
\phantom{\includegraphics[width=0.2\textwidth]{EOL_Pareto_bStarZeta}} \\
b^*(\rho) &&  b^*(\zeta) && \\ 
\end{array}
\end{eqnarray*}
\caption{Dependency of $b^*$ with respect to $\rho$ and $\zeta$.}\label{fig5}
\end{figure}

\begin{figure}[htb!]
\begin{eqnarray*}
\begin{array}{ccccccc}
\includegraphics[width=0.27\textwidth]{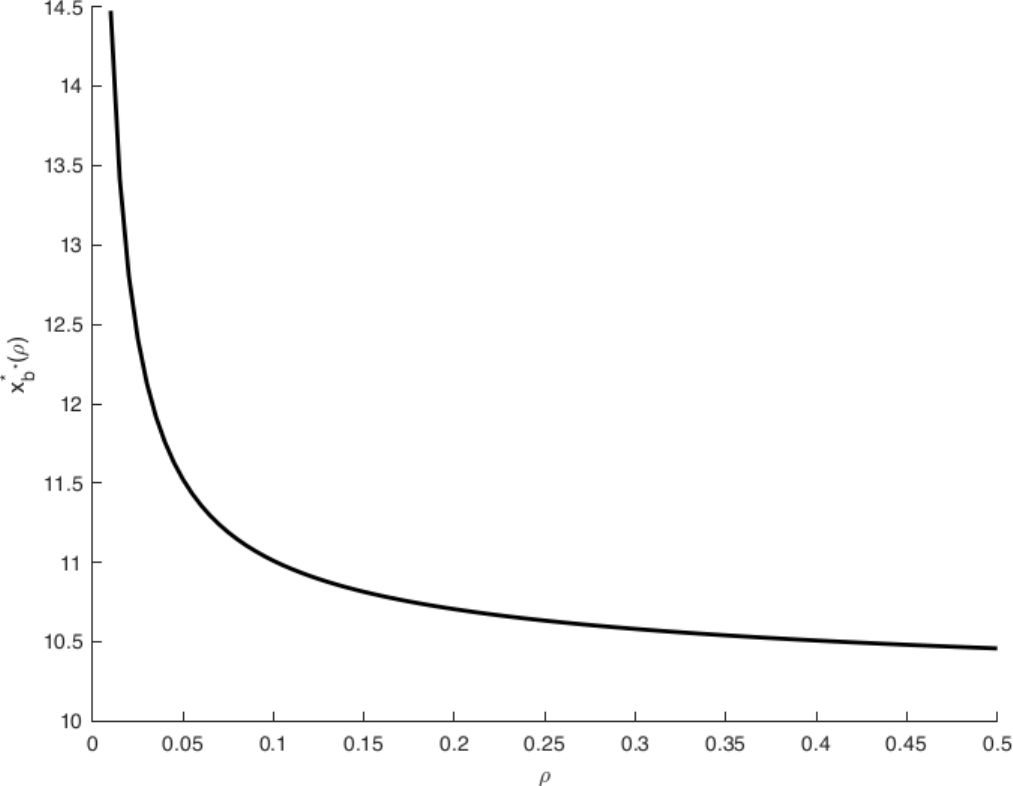} &&
\includegraphics[width=0.27\textwidth]{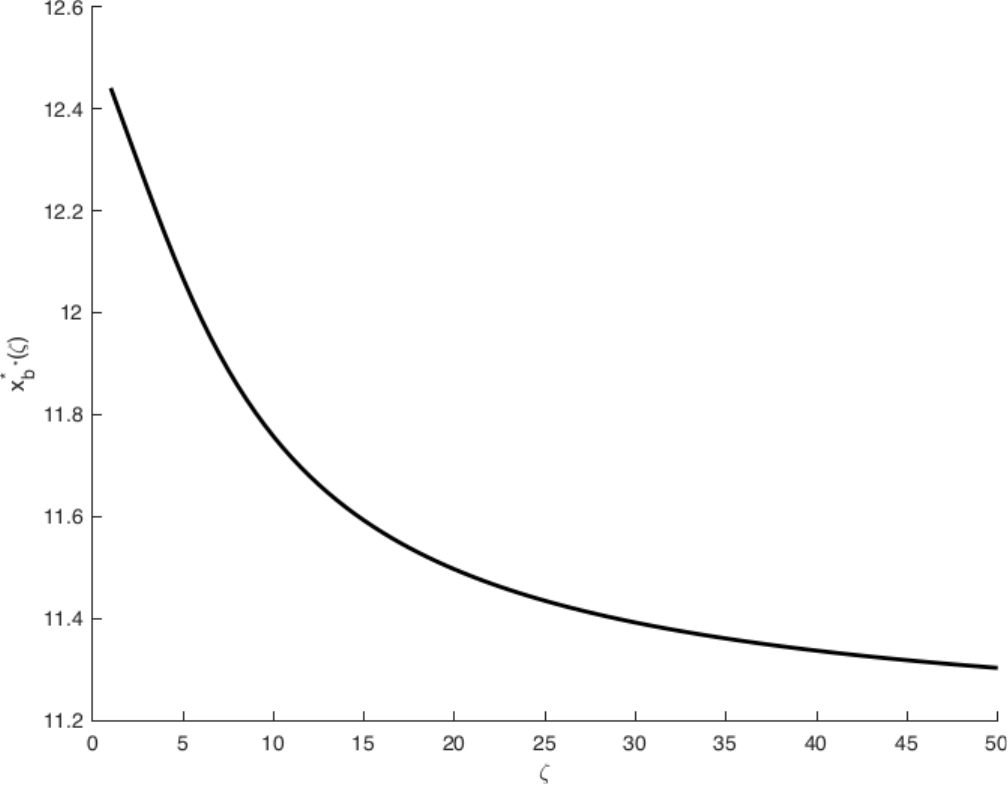} &&
\includegraphics[width=0.27\textwidth]{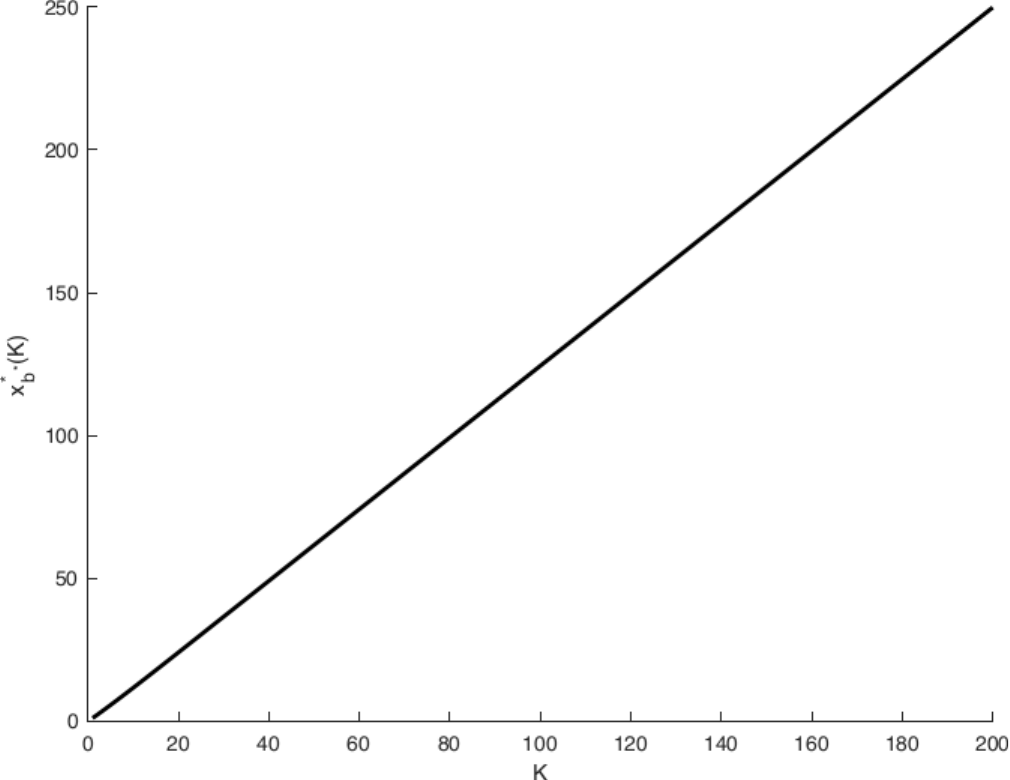}\\
x^*_{b^*}(\rho) &&  x^*_{b^*}(\zeta) && x^*_{b^*}(K)\\ 
\end{array}
\end{eqnarray*}
\caption{Dependency of $x^*_{b^*}$ with respect to $\rho$, $\zeta$ and $K$.}\label{fig6}
\end{figure}

\subsection{Comparison of the value function in the case of proportional and excess-of-loss reinsurance}

In Figure \ref{fig7} we collect drawings of the value function in the case of proportional reinsurance (solid line) and excess-of-loss reinsurance (dashed line),
when the claim size $Z \sim Exp(1/\mu)$ (left panel) and $Z \sim Pareto(\zeta,\alpha)$ (right panel). We observe that, when the claim's size is exponentially distributed, the value function one obtains in the case of proportional reinsurance is smaller than the one related to an excess-of-loss reinsurance, while no uniform comparison can be made in the case of Pareto-distributed claim's size. We thus conclude that (at least for the benchmark values of the parameters that we have used) excess-of-loss reinsurance is not necessarily favourable to proportional reinsurance, a finding that is in contrast to that of Figure 3 in \cite{EisenbergSchmidli2009} (see also the subsequent discussion at page 13 therein).

\begin{figure}[htb!]
\begin{eqnarray*}
\begin{array}{ccc}
\includegraphics[width=0.40\textwidth]{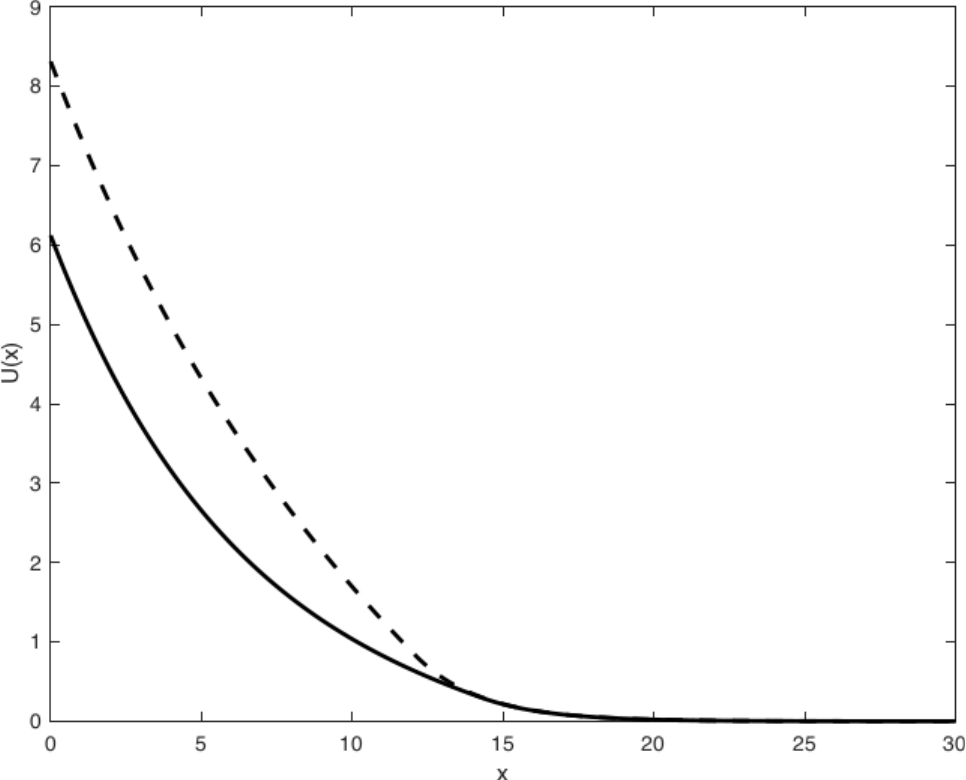} &&
\includegraphics[width=0.40\textwidth]{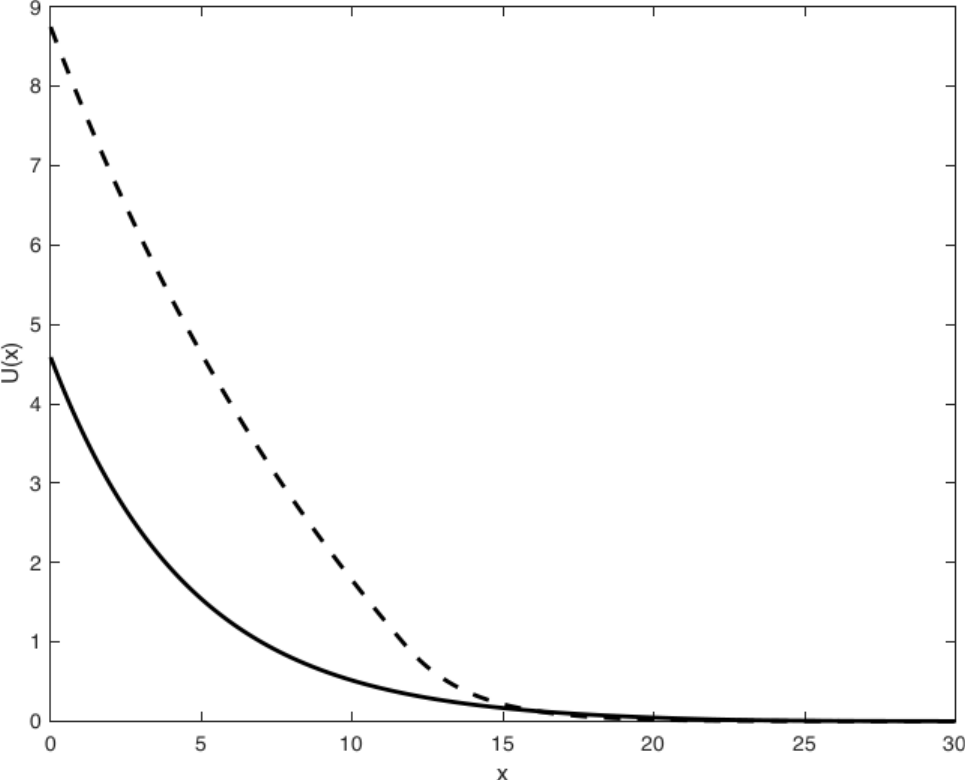}\\
\includegraphics[width=0.40\textwidth]{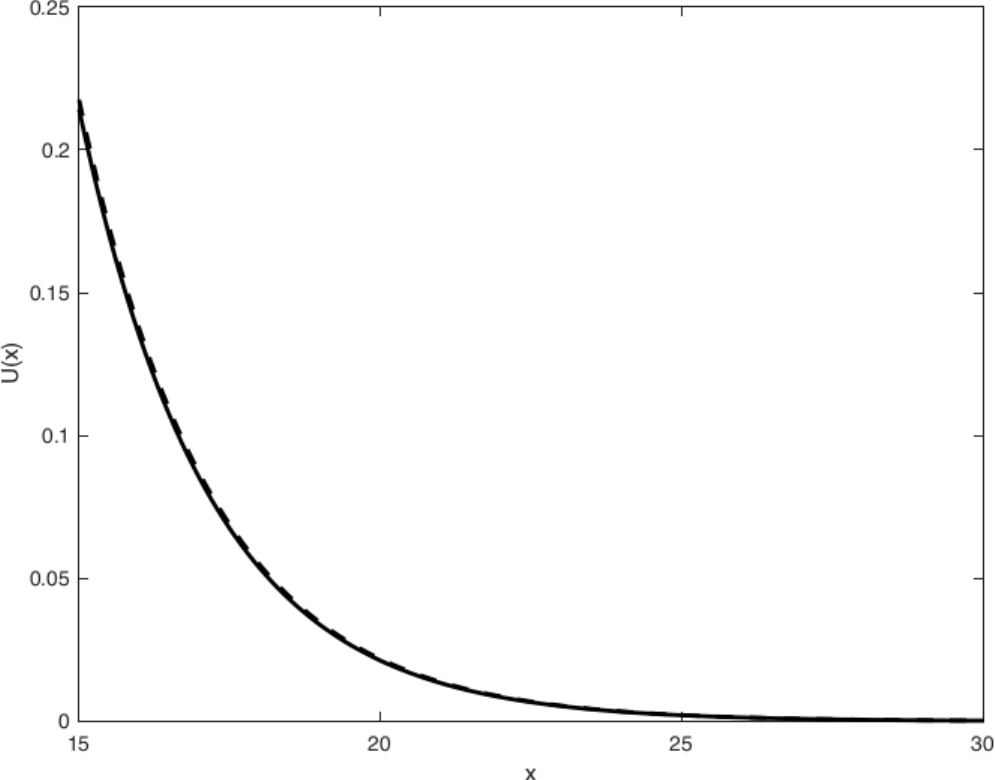} &&
\includegraphics[width=0.40\textwidth]{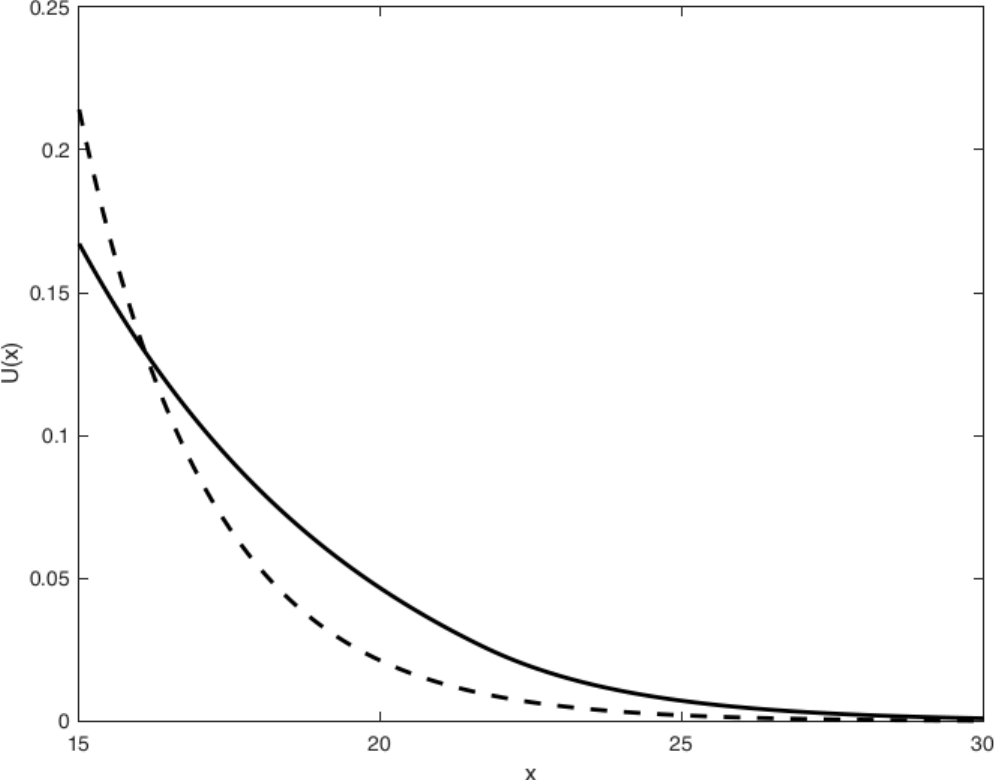}\\
Z \sim Exp(1/\mu) &&  Z \sim Pareto(\zeta,\alpha)\\ 
\end{array}
\end{eqnarray*}
\caption{Value function (zoomed image in the bottom panels) in the case of proportional reinsurance (solid line) and excess-of-loss reinsurance 
(dashed line), when the claim sizes $Z \sim Exp(1/\mu)$ (left panel) and $Z \sim Pareto(\zeta,\alpha)$ (right panel).}\label{fig7}
\end{figure}


\section*{Proof of Proposition \ref{propfb}}
\begin{proof}
From \eqref{GbxAnalyticPos}, \eqref{GbxAnalyticNeg} and \eqref{fb}, the explicit expression of $f_{b^*}$ is:
\begin{eqnarray}\label{fbstar}
f_{b^*}(x) = 
\begin{cases}
\displaystyle{-(x-K)-\frac{1}{\gamma^-(b^*)}+ \frac{1}{\gamma^-(1)}e^{\gamma^-(1)x} }& 0 \le x\le K,\\
\displaystyle{-\frac{1}{\gamma^-(b^*)}e^{\gamma^-(b^*)(x-K)}+\frac{1}{\gamma^-(1)}e^{\gamma^-(1)x}} & x > K,
\end{cases}
\end{eqnarray}
from which $f_{b^*}(0)=\displaystyle{K+ \frac{\gamma^-(b^*) - \gamma^-(1)}{\gamma^-(1) \gamma^-(b^*)}}$ and $\lim_{x \to \infty} f_{b^*}(x)=0$. 
We compute 
\begin{eqnarray*}
f'_{b^*}(x) = 
\begin{cases}
\displaystyle{-1+e^{\gamma^-(1)x} }& 0 < x< K,\\
\displaystyle{-e^{\gamma^-(b^*)(x-K)}+e^{\gamma^-(1)x}} & x > K.
\end{cases}
\end{eqnarray*}
Then, $f_{b^*}$ is strictly decreasing in $[0,\hat x_{b^*}]$ and strictly increasing in $[\hat x_{b^*},\infty)$, where $\hat x_{b^*}$ is defined in \eqref{xHat}.
The point $\hat x_{b^*}$ is the unique global minimum point of $f_{b^*}$, whose minimum value is:
$$
f_{b^*}(\hat x_{b^*})=\displaystyle{-\frac{\gamma^-(1)-\gamma^-(b^*)}{\gamma^-(1)\gamma^-(b^*)}e^{-\frac{\gamma^-(1)\gamma^-(b^*)}{\gamma^-(1)-\gamma^-(b^*)}}}.
$$ 

If $-K \gamma^-(1) \gamma^-(b^*) \le\gamma^-(b^*) - \gamma^-(1)<0$, then $f_{b^*}(0)\ge 0$ and item (i) follows. 
Otherwise, if $\displaystyle{\gamma^-(b^*) - \gamma^-(1)}<-K \gamma^-(1) \gamma^-(b^*)$, then $f_{b^*}(0)< 0$ and item (ii) follows.
\end{proof}

\section*{Proof of Theorem \ref{th:ver}}
\begin{proof}
Let $x\ge0$, $T > 0$, $\tau_n:= \inf\{t \ge 0 \ : \ X^x_t \ge n\}$, $n\geq 0$, and $\tau \in \mathcal T$.
Applying a change of variable formula for semimartingales (see e.g.~\cite{CaiDeAngelis}, Theorem 2.1) to $\{e^{-\rho t}w(X^x_t), t \in [0, \tau_n \wedge \tau \wedge T]\}$ and then taking expectations we find:
$$
\E\left[ e^{-\rho(\tau_n \wedge \tau \wedge T)} w(X^x_{\tau_n \wedge \tau \wedge T})\right]=
w(x) + \E\left[\int_0^{\tau_n \wedge \tau \wedge T} e^{-\rho s} \left(\left ( \mathcal L - \rho \right) w\right) (X^x_s) ds
+\int_0^{\tau_n \wedge \tau \wedge T} e^{-\rho s} w'(X^x_s) dI_s\right],
$$
where the Brownian-local martingale term has vanished in expectation since $w \in C^1([0,\infty); \R)$ (by Sobolev embedding) and 
because of the definition of $\tau_n$. With regards to the fact that $w$ solves \eqref{VI} and $t \mapsto I_t$ increases on $\{t \ge 0 \ : \ X^x_t = 0\}$,
we have (after rearranging terms)
$$
w(x) \le \E\left[ e^{-\rho(\tau_n \wedge \tau \wedge T)} f_{b^*} (X^x_{\tau_n \wedge \tau \wedge T})\right ].
$$ 
As $f_{b^*}$ is bounded (cf.\ Proposition \ref{propfb}), by sending $n \uparrow \infty$ and $T \uparrow \infty$, by the dominated convergence theorem we obtain
$$
w(x) \le \E\left[ e^{-\rho\tau} f_{b^*} (X^x_\tau)\right ].
$$
Given the arbitrariness of $\tau \in \mathcal{T}$ and $x\geq 0$, we have 
$$
w \le F_{b^*} \textrm{ on } [0,\infty).
$$
Repeating now the same arguments above, but with $\tau$ replaced by $\tau^*(b^*)$, we find (by definition of $\tau^*(b^*)$)
$$
w(x) = \E\left[ e^{-\rho\tau^*(b^*)} f_{b^*} \left(X^x_{\tau^*(b^*)}\right)\right ].
$$
Hence,
$$
w(x) \ge \inf_{\tau \in \mathcal T} \E\left[ e^{-\rho\tau} f_{b^*} (X^x_{\tau})\right ] = F_{b^*}(x).
$$
Given the arbitrariness of $x \ge 0$, we have $w \ge F_{b^*}$, which, together with the previously proved $w \le F_{b^*}$, 
implies that $w=F_{b^*}$ and that $\tau^*(b^*)$ is optimal.
\end{proof}

\section*{Proof of Proposition \ref{prop1}}
\begin{proof}
{\bf Step 1.} We here prove existence and uniqueness of $x^*_{b^*}$.
Using \eqref{functionH}, we rewrite problem \eqref{FB} as follows
\begin{eqnarray}\label{freeBound}
\left\{\begin{array}{l}
C_1\gamma^-(1)+C_2\gamma^+(1)=0\\
C_1 e^{\gamma^-(1)x}+C_2 e^{\gamma^+(1)x}=f_{b^*}(x)\\
C_1 \gamma^-(1) e^{\gamma^-(1)x}+C_2 \gamma^+(1) e^{\gamma^+(1)x}=f_{b^*}'(x),
\end{array}\right.
\end{eqnarray}
where the explicit expression of $f_{b^*}$ is given in \eqref{fbstar}. The first equation yields:
\begin{eqnarray}\label{expressionC1C2}
C_2=-C_1\frac{\gamma^-(1)}{\gamma^+(1)},
\end{eqnarray}
whereas the second and third equation depend on the expression of $f_{b^*}$. We consider two distinct cases.\\
(i) If $0 \le x \le K$, then the second and the third equations of \eqref{freeBound} become
\begin{eqnarray}\label{system1}
\left\{\begin{array}{l}
C_1  \gamma^-(1)  \left(e^{\gamma^-(1)x} -\frac{\gamma^-(1)}{\gamma^+(1)}e^{\gamma^+(1)x}\right)=
-\gamma^-(1) (x-K)-\frac{\gamma^-(1) }{\gamma^-(b^*)}+ e^{\gamma^-(1)x}\\
C_1 \gamma^-(1) \left(e^{\gamma^-(1)x}-  e^{\gamma^+(1)x}\right)=-1+ e^{\gamma^-(1)x}.
\end{array}\right.
\end{eqnarray}
Since $\gamma^-(1) \neq \gamma^+(1)$, the previous system yields:
\begin{eqnarray*}
\frac{-\gamma^-(1) (x-K)-\frac{\gamma^-(1) }{\gamma^-(b^*)}+ e^{\gamma^-(1)x}}{e^{\gamma^-(1)x} -\frac{\gamma^-(1)}{\gamma^+(1)}e^{\gamma^+(1)x}}&=&
\frac{-1+ e^{\gamma^-(1)x}}{e^{\gamma^-(1)x}-  e^{\gamma^+(1)x}},
\end{eqnarray*}
which can be rewritten as
\begin{eqnarray}\label{eqToSolve1}
F_1(x)=F_2(x),
\end{eqnarray}
where
\begin{eqnarray*}
\begin{array}{lll}
F_1(x) &:=& \displaystyle{\left(x-K+\frac{1}{\gamma^-(b^*)}\right)\left(e^{-\gamma^-(1) x} - e^{-\gamma^+(1) x}\right)}\\
F_2(x) &:=& \displaystyle{\frac{1}{\gamma^-(1)}\left(1-e^{-\gamma^+(1) x}\right)-\frac{1}{\gamma^+(1)}\left(1-e^{-\gamma^-(1) x}\right)}.
\end{array}
\end{eqnarray*}
We have $F_1(0)=F_2(0)=0$ and
\begin{eqnarray*}
F_1'(x)&=&\gamma^+(1)\left(x-K+\frac{1}{\gamma^-(b^*)}-\frac{1}{\gamma^+(1)}\right)e^{-\gamma^+(1) x}
-\gamma^-(1)\left(x-K+\frac{1}{\gamma^-(b^*)}-\frac{1}{\gamma^-(1)}\right)e^{-\gamma^-(1) x}\\
F_2'(x)&=&\frac{\gamma^+(1)}{\gamma^-(1)}e^{-\gamma^+(1) x}-\frac{\gamma^-(1)}{\gamma^+(1)}e^{-\gamma^-(1) x}.
\end{eqnarray*}
From \eqref{pol1} we note that 
$$
\frac{1}{\gamma^-(1)}+\frac{1}{\gamma^+(1)} = \frac{\gamma^-(1) + \gamma^+(1)}{\gamma^-(1) \gamma^+(1)} = \frac{\lambda \eta \mu}{\rho}>0.
$$
Consequently $\displaystyle{\frac{1}{\gamma^-(b^*)}<\frac{1}{\gamma^+(1)}+\frac{1}{\gamma^-(1)}}$ and, for each $x \in [0,K]$, it holds: 
\begin{eqnarray*}
F_1'(x)&=&\gamma^+(1)\left(x-K+\frac{1}{\gamma^-(b^*)}-\frac{1}{\gamma^+(1)}\right)e^{-\gamma^+(1) x}
-\gamma^-(1)\left(x-K+\frac{1}{\gamma^-(b^*)}-\frac{1}{\gamma^-(1)}\right)e^{-\gamma^-(1) x}\\
&\le& \gamma^+(1)\left(\frac{1}{\gamma^-(b^*)}-\frac{1}{\gamma^+(1)}\right)e^{-\gamma^+(1) x}
-\gamma^-(1)\left(\frac{1}{\gamma^-(b^*)}-\frac{1}{\gamma^-(1)}\right)e^{-\gamma^-(1) x} \\
&<& \frac{\gamma^+(1)}{\gamma^-(1)}e^{-\gamma^+(1) x}-\frac{\gamma^-(1)}{\gamma^+(1)}e^{-\gamma^-(1) x} = F_2'(x).
\end{eqnarray*} 
Hence, the unique solution of \eqref{eqToSolve1} in $[0,K]$ is $x=0$, and \eqref{system1} and \eqref{expressionC1C2} yield $C_1=C_2=0$.\\

\noindent
(ii) If $x >K$, then the second and the third equation of \eqref{freeBound} become
\begin{eqnarray}\label{system2}
\left\{\begin{array}{l}
C_1  \gamma^-(1)  \left(e^{\gamma^-(1)x} -\frac{\gamma^-(1)}{\gamma^+(1)}e^{\gamma^+(1)x}\right)=
-\frac{\gamma^-(1)}{\gamma^-(b)}e^{\gamma^-(b^*)(x-K)}+e^{\gamma^-(1)x}\\
C_1 \gamma^-(1) \left(e^{\gamma^-(1)x}-  e^{\gamma^+(1)x}\right)=-e^{\gamma^-(b^*)(x-K)}+e^{\gamma^-(1)x}.
\end{array}\right.
\end{eqnarray}
Since $\gamma^-(1) \neq \gamma^+(1)$,  the previous system yields:
\begin{eqnarray*}
\frac{-\frac{\gamma^-(1)}{\gamma^-(b^*)}e^{\gamma^-(b^*)(x-K)}+e^{\gamma^-(1)x}}{e^{\gamma^-(1)x} -\frac{\gamma^-(1)}{\gamma^+(1)}e^{\gamma^+(1)x}}
&=&
\frac{-e^{\gamma^-(b^*)(x-K)}+e^{\gamma^-(1)x}}{e^{\gamma^-(1)x}-  e^{\gamma^+(1)x}},
\end{eqnarray*}
which can be rewritten as \eqref{eqToSolve}; that is,
\begin{eqnarray}\label{eqToSolveBis}
\Theta(x)=D(K),
\end{eqnarray}
where
\begin{eqnarray*}
\Theta(x):= \gamma^+(1)(\gamma^-(b^*)-\gamma^-(1))e^{(\gamma^-(b^*)-\gamma^+(1))x}-\gamma^-(1)(\gamma^-(b^*)-\gamma^+(1))e^{(\gamma^-(b^*)-\gamma^-(1))x}
\end{eqnarray*}
and 
\begin{eqnarray*}
D(K):=\gamma^-(b^*)(\gamma^+(1)-\gamma^-(1))e^{\gamma^-(b^*)K} <0.
\end{eqnarray*}
We have 
\begin{eqnarray*}
\Theta(0)=\gamma^-(b^*)(\gamma^+(1)-\gamma^-(1)) \le 
\gamma^-(b^*)(\gamma^+(1)-\gamma^-(1))e^{\gamma^-(b^*)K} =D(K).
\end{eqnarray*}
and 
\begin{eqnarray*}
\Theta'(x)= (\gamma^-(b^*)-\gamma^-(1))(\gamma^-(b^*)-\gamma^+(1))\left(\gamma^+(1)e^{-\gamma^+(1)x}-\gamma^-(1)e^{-\gamma^-(1)x}\right)e^{\gamma^-(b^*)x}>0.
\end{eqnarray*}
Since $\Theta$ is strictly increasing and $\lim_{x \to +\infty} \Theta(x)=0$, \eqref{eqToSolveBis} (and \eqref{eqToSolve}) admits in $[0,+\infty)$ the unique positive 
solution $x^*_{b^*}=x^*_{b^*}(K)=\Theta^{-1}(D(K))>0$. Further, since $D'(K)=(\gamma^-(b^*))^2(\gamma^+(1)-\gamma^-(1))e^{\gamma^-(b^*)K}>0$, then 
$x^*_{b^*}$ is strictly increasing in $K$.\\

\noindent
{\bf Step 2.} We show that $x^*_{b^*} \le \hat x_{b^*}$. Because
\begin{eqnarray*}
\Theta(\hat x_{b^*})&=& \gamma^+(1)(\gamma^-(b^*)-\gamma^-(1))e^{\frac{\gamma^-(b^*)-\gamma^+(1)}{\gamma^-(b^*)-\gamma^-(1)}\gamma^-(b^*) K}
-\gamma^-(1)(\gamma^-(b^*)-\gamma^+(1))e^{\gamma^-(b^*) K}\\
&=& \left(\gamma^+(1)(\gamma^-(b^*)-\gamma^-(1)) e^{\frac{\gamma^-(1)-\gamma^+(1)}{\gamma^-(b^*)-\gamma^-(1)}\gamma^-(b^*) K}
-\gamma^-(1)(\gamma^-(b^*)-\gamma^+(1))\right)e^{\gamma^-(b^*) K}\\
&\ge& \left(\gamma^+(1)(\gamma^-(b^*)-\gamma^-(1))-\gamma^-(1)(\gamma^-(b^*)-\gamma^+(1))\right)e^{\gamma^-(b^*) K}\\
&=& \gamma^-(b^*)(\gamma^+(1)-\gamma^-(1))e^{\gamma^-(b^*) K}=D(K),
\end{eqnarray*}
then the monotonicity of $\Theta$ yields $x^*_{b^*}\le\hat x_{b^*}=\displaystyle{\frac{\gamma^-(b^*)}{\gamma^-(b^*) - \gamma^-(1) }}K$. \\

\noindent
{\bf Step 3.} We now aim at proving that $x^*_{b^*} \ge K$. Notice that
\begin{eqnarray*}
\Theta(K)&=& e^{\gamma^-(b^*) K}\left(\gamma^+(1)(\gamma^-(b^*)-\gamma^-(1))e^{-\gamma^+(1)K}-\gamma^-(1)(\gamma^-(b^*)-\gamma^+(1))e^{-\gamma^-(1) K}\right).
\end{eqnarray*}
Let
$$
S(K):=\gamma^+(1)(\gamma^-(b^*)-\gamma^-(1))e^{-\gamma^+(1)K}-\gamma^-(1)(\gamma^-(b^*)-\gamma^+(1))e^{-\gamma^-(1) K}, \;\;\; K\ge 0.
$$
It holds $S(0)=\gamma^-(b^*)(\gamma^+(1)-\gamma^-(1))<0$ and
\begin{eqnarray*}
S'(K)&=&-\gamma^+(1)^2(\gamma^-(b^*)-\gamma^-(1))e^{-\gamma^+(1)K}+\gamma^-(1)^2(\gamma^-(b^*)-\gamma^+(1))e^{-\gamma^-(1) K}\\
&=& -\gamma^+(1)^2(\gamma^-(b^*)-\gamma^-(1))e^{-\gamma^-(1) K}\left(e^{(\gamma^-(1)-\gamma^+(1))K}
-\frac{\gamma^-(1)^2\left(\gamma^-(b^*)-\gamma^+(1)\right)}{\gamma^+(1)^2 \left(\gamma^-(b^*)-\gamma^-(1)\right)}\right).
\end{eqnarray*}
Since $\gamma^-(1)+\gamma^+(1)<0$, then $\frac{\gamma^-(1)^2\left(\gamma^-(b^*)-\gamma^+(1)\right)}{\gamma^+(1)^2 \left(\gamma^-(b^*)-\gamma^-(1)\right)}>1$.
Therefore, $S'(K)<0$ for each $K \ge 0$, so that 
$$
\Theta(K) = e^{\gamma^-(b^*) K} S(K)\le e^{\gamma^-(b^*) K} S(0)=\gamma^-(b^*)(\gamma^+(1)-\gamma^-(1))e^{\gamma^-(b^*) K}=D(K)
$$
and $K\le x^*_{b^*}$ by monotonicity of $\Theta$.\\


\noindent
{\bf Step 4.} Finally, using \eqref{system2} and \eqref{expressionC1C2} we get $C_1=\frac{1}{\gamma^-(1)}B$ and $C_2=\frac{1}{\gamma^+(1)}B$, where 
\begin{eqnarray}\label{expressionB1}
B=\displaystyle{\frac{e^{\gamma^-(b^*)(x^*_{b^*}-K)}-e^{\gamma^-(1) x^*_{b^*}}}{e^{\gamma^+(1) x^*_{b^*}}-e^{\gamma^-(1) x^*_{b^*}}}}. 
\end{eqnarray}
We rewrite \eqref{eqToSolve} (or equivalently $\Theta(x)=D(K)$) as follows:
\begin{eqnarray}\label{equivalent}
e^{\gamma^-(b^*)(x^*_{b^*}-K)} H(x^*_{b^*})= \gamma^{-}(b^*)\left(\gamma^+(1)-\gamma^-(1)\right)e^{(\gamma^-(1)+\gamma^+(1))x^*_{b^*}},
\end{eqnarray}
where $H(x^*_{b^*})$ is given in \eqref{expressionH}.
Using \eqref{equivalent} in \eqref{expressionB1} we get
\begin{eqnarray*}
B &=& \displaystyle{\frac{e^{\gamma^-(b^*)(x^*_{b^*}-K)}-e^{\gamma^-(1) x^*_{b^*}}}{e^{\gamma^+(1) x^*_{b^*}}-e^{\gamma^-(1) x^*_{b^*}}}}\\
&=& \displaystyle{\left( \frac{\gamma^{-}(b^*)\left(\gamma^+(1)-\gamma^-(1)\right)e^{\gamma^+(1)x^*_{b^*}}}{H(x^*_{b^*})} -1 \right)
\frac{e^{\gamma^-(1) x^*_{b^*}} } {e^{\gamma^+(1) x^*_{b^*}}-e^{\gamma^-(1) x^*_{b^*}} } }\\
&=& \displaystyle{ 
\frac{\gamma^+(1)(\gamma^-(b)-\gamma^-(1))(e^{\gamma^+(1) x^*_{b^*}}-e^{\gamma^-(1) x^*_{b^*}} )}{H(x^*_{b^*})} 
\frac{e^{\gamma^-(1) x^*_{b^*}} } {e^{\gamma^+(1) x^*_{b^*}}-e^{\gamma^-(1) x^*_{b^*}} } }\\
&=& \displaystyle{ \gamma^+(1)(\gamma^-(b)-\gamma^-(1))
\frac{e^{\gamma^-(1) x^*_{b^*}}}{H(x^*_{b^*})}   =: B(x^*_{b^*})}
\end{eqnarray*}
as given by \eqref{expressionB}. 
Since $\gamma^-(b)<\gamma^-(1)$ and $H<0$, then $B(x^*_{b^*}) >0$; further, $H<\gamma^+(1)(\gamma^-(b)-\gamma^-(1))e^{\gamma^-(1) x^*_{b^*}}$ and
therefore $B(x^*_{b^*})<1$.
\end{proof}

\section*{Proof of Theorem \ref{prop2}}
\begin{proof}
In order to prove that $w \equiv V$, we need to show that:
\begin{eqnarray}\label{uEquivU}
\left\{\begin{array}{lll}
w(x) \le f_{b^*}(x), &&\forall x > 0\\
w \textrm{ is s.t. } 0 \le (\mathcal L-\rho)w(x), &&\forall x > 0,
\end{array}\right. 
\end{eqnarray}
which is implied by the following two conditions:
\begin{eqnarray*}
\begin{array}{lllll}
\textrm{(i)} && w(x) \le f_{b^*}(x), &&\forall 0< x \le x^*_{b^*}\\
\textrm{(ii)} && 0 \le (\mathcal L-\rho)w(x), &&\forall x > x^*_{b^*}.
\end{array}
\end{eqnarray*}

\noindent
We start proving (i). By Proposition \ref{prop1} 
$$
w(x) = \displaystyle{ B(x^*_{b^*})\left(\frac{1}{\gamma^-(1)} e^{\gamma^-(1)x} - \frac{1}{\gamma^+(1)}e^{\gamma^+(1)x} \right)},
$$
where $B(x^*_{b^*})$ is given by \eqref{expressionB}. 
We compute $w(0)=B(x^*_{b^*})\left(\frac{1}{\gamma^-(1)} - \frac{1}{\gamma^+(1)}\right)$ and $f_{b^*}(0) = K-\frac{1}{\gamma^-(b^*)}+\frac{1}{\gamma^-(1)}$, 
so that $w(0)\le f_{b^*}(0)$ if and only if 
\begin{eqnarray*}
K &\ge& \displaystyle{B(x^*_{b^*})\left(\frac{1}{\gamma^-(1)} - \frac{1}{\gamma^+(1)}\right)+\frac{1}{\gamma^-(b^*)}-\frac{1}{\gamma^-(1)}}\\
&=& \displaystyle{\frac{(\gamma^-(b^*)-\gamma^+(1))(\gamma^-(b^*)-\gamma^-(1))}{\gamma^-(b^*)}\frac{e^{\gamma^+(1)x^*_{b^*}} - e^{\gamma^-(1)x^*_{b^*}}}{H(x^*_{b^*})}}
=Q(x^*_{b^*}),
\end{eqnarray*}
where, for $x>0$,
$$
Q(x):= \displaystyle{\frac{(\gamma^-(b^*)-\gamma^+(1))(\gamma^-(b^*)-\gamma^-(1))}{\gamma^-(b^*)}\frac{e^{\gamma^+(1)x} - e^{\gamma^-(1)}}{H(x)}},
$$
and (cf.~\eqref{expressionH})
$$
H(x):=\gamma^+(1)\left(\gamma^-(b^*)-\gamma^-(1)\right)e^{\gamma^-(1)x} - \gamma^-(1)\left(\gamma^-(b^*)-\gamma^+(1)\right)e^{\gamma^+(1)x}<0.
$$
By some computations we have
$$
Q'(x)=\displaystyle{(\gamma^-(b^*)-\gamma^+(1))(\gamma^-(b^*)-\gamma^-(1))(\gamma^+(1)-\gamma^-(1))^2
\frac{e^{(\gamma^+(1)+\gamma^-(1))x} }{(H(x))^2}}>0.
$$
Recalling that $x^*_{b^*}=x^*_{b^*}(K)=\Theta^{-1}(D(K))$ (see Step 1 in the proof of Proposition \eqref{prop1}), it holds 
$Q(x^*_{b^*}(0))=Q(0)=0$.
Further,
$$
\frac{\partial Q(x^*_{b^*}(K))}{\partial K}=Q'(x^*_{b^*}(K)) \cdot (x^*_{b^*})'(K).
$$
Since $x^*_{b^*}$ is strictly increasing in $K$ it follows that $Q(x^*_{b^*}(K))$ is strictly increasing in $K$ and
\begin{eqnarray*}
\frac{\partial Q(x^*_{b^*}(0))}{\partial K}&=&Q'(0) \cdot (x^*_{b^*})'(0)\\ 
&=&\displaystyle{\frac{(\gamma^-(b^*)-\gamma^+(1))(\gamma^-(b^*)-\gamma^-(1))}{\gamma^-(b^*)^2}\cdot
\frac{\gamma^-(b^*)^2}{(\gamma^-(b^*)-\gamma^-(1))(\gamma^-(b^*)-\gamma^+(1))}}=1.
\end{eqnarray*}
In order to study the concavity of $K \mapsto Q(x^*_{b^*}(K))$, we compute 
\begin{eqnarray}\label{secondDeriv}
\frac{\partial^2 Q(x^*_{b^*}(K))}{\partial K^2}=
Q''(x^*_{b^*}(K))\cdot \left((x^*_{b^*})'(K)\right)^2+Q'(x^*_{b^*}(K))\cdot (x^*_{b^*})''(K).
\end{eqnarray}
We find
\begin{eqnarray*}
Q''(x)=
\displaystyle{(\gamma^-(b^*)-\gamma^+(1))(\gamma^-(b^*)-\gamma^-(1))(\gamma^+(1)-\gamma^-(1))^3
\frac{e^{(\gamma^+(1)+\gamma^-(1))x} }{(H(x))^3}}A(x)
\end{eqnarray*}
where 
\begin{eqnarray*}
A(x):=\gamma^+(1)(\gamma^-(b^*)-\gamma^-(1)) e^{\gamma^-(1)x}+ \gamma^-(1)(\gamma^-(b^*)-\gamma^+(1)) e^{\gamma^+(1)x}.
\end{eqnarray*}
Since $\gamma^+(1)+\gamma^-(1)=-\displaystyle{\frac{2\eta \mu}{\sigma^2}} <0$, then $A(x^*_{b^*}(0))=A(0)=\gamma^-(b^*)(\gamma^+(1)+\gamma^-(1)) -2 \gamma^-(1)\gamma^+(1)>0$; further, 
$$
A'(x)=\gamma^+(1)\gamma^-(1)\left((\gamma^-(b^*)-\gamma^-(1)) e^{\gamma^-(1)x}+(\gamma^-(b^*)-\gamma^+(1)) e^{\gamma^+(1)x}\right)
>0,
$$
therefore $A(x)\ge 0$ for each $x\ge 0$, and consequently, since $H(x) <0$ for each $x\ge 0$, it holds $Q''(x)\le 0$. In particular, 
\begin{eqnarray}\label{deriv1}
Q''(x^*_{b^*}(K))\le 0, \ \ \ \ \textrm{ for each } K \ge 0.
\end{eqnarray}
Using that $x^*_{b^*}=x^*_{b^*}(K)$ is the unique solution to the equation $\Psi(x,K)=\Theta(x)-D(K)=0$ (see again step 1 of the proof of Proposition \ref{prop1}), 
we have (see \cite{BrockThompson}) 
$(x^*_{b^*})''(K)=\displaystyle{\left(\frac{\partial \Psi}{\partial K}\frac{\partial^2 \Psi}{\partial K \partial x} -\frac{\partial \Psi}{\partial x}\frac{\partial^2 \Psi}{\partial K^2}\right) /\left(\frac{\partial \Psi}{\partial x}\right)^2}$, 
so that the sign of $(x^*_{b^*})''(K)$ coincides with the sign of the following quantity:
\begin{eqnarray}\label{deriv2}
\displaystyle{\frac{\partial \Psi}{\partial K}\frac{\partial^2 \Psi}{\partial K \partial x} -\frac{\partial \Psi}{\partial x}\frac{\partial^2 \Psi}{\partial K^2} }=
\Theta'(x) \cdot D''(K) <0.
\end{eqnarray}
Using \eqref{secondDeriv}, \eqref{deriv1} and \eqref{deriv2}, we get $\frac{\partial^2 Q(x^*_{b^*}(K)}{\partial K^2}<0$;
consequently, $K \ge Q(x^*_{b^*}(K))$, and $w(0)\le f_{b^*}(0)$.
Further, $w$ is negative and, because 
$$
w'(x)=B\left(e^{\gamma^-(1)x} -e^{\gamma^+(1)x}\right) <0,
$$
$w$ is strictly decreasing. We recall by Proposition \ref{propfb} that $f_{b^*}$ is strictly decreasing in $[0,\hat x_{b^*}]$.
If $x \in [0,K]$ then 
$$
w'(x)=B\left(e^{\gamma^-(1)x} -e^{\gamma^+(1)x}\right) < e^{\gamma^-(1)x} -e^{\gamma^+(1)x} < e^{\gamma^-(1)x} -1=f_{b^*}'(x).
$$
Consequently, it holds $w(x) \le f_{b^*}(x)$ for each $x \in [0,K]$. 
On the other hand, since $\displaystyle{x^*_{b^*}>K>\frac{\gamma^-(b)K}{\gamma^-(b)-\gamma^+(1)}}$, then for each $x \in [K,x^*_{b^*}]$
$$
w'(x)=B\left(e^{\gamma^-(1)x} -e^{\gamma^+(1)x}\right) < e^{\gamma^-(1)x} -e^{\gamma^+(1)x} < e^{\gamma^-(1)x} -e^{\gamma^-(b)(x-K)}=f_{b^*}'(x),
$$
which, coupled with $w(K)<f_{b^*}(K)$, yields $w(x) \le f_{b^*}(x)$ for each $x \in [K,x^*_{b^*}]$.

\bigskip
We now prove (ii). For each $x>x^*_{b^*}>K$ we have
\begin{eqnarray}\label{conto}
\ \ \ (\mathcal L-\rho)f_{b^*}(x)&=&  (\mathcal L-\rho)\left(-\frac{1}{\gamma^-(b^*)}e^{\gamma^-(b^*)(x-K)}+\frac{1}{\gamma^-(1)}e^{\gamma^-(1)x}\right) \nonumber\\
&=& -\frac{1}{\gamma^-(b^*)}e^{\gamma^-(b^*)(x-K)}\left(\frac{1}{2}\lambda \sigma^2(\gamma^-(b^*))^2
+\lambda(\theta \mu - (\theta - \eta)\mu)\gamma^-(b^*)-\rho\right) \nonumber \\
&=& -\frac{1}{\gamma^-(b^*)}e^{\gamma^-(b^*)(x-K)} \Phi(1,\gamma^-(b^*)).
\end{eqnarray}
Since $\gamma^-(1)$ is the unique negative solution to $\Phi(1,\gamma)=0$ and because $\gamma^-(b^*)<\gamma^-(1)$, 
then $\Phi(1,\gamma^-(b))>0$ and, by \eqref{conto}, $(\mathcal L-\rho)f_{b^*}(x)>0$ for each $x>x^*_{b^*}$.
\end{proof}

\section*{Proof of Proposition \ref{propInternalMin}}
\begin{proof}
From \eqref{polProp}, by some computations, the derivative $(\gamma^-)'$ of $\gamma^-$ is
\begin{eqnarray}\label{gammaPrimeProp}
(\gamma^-)'(b)=-\gamma^-(b)\frac{\phi(b)}{D(b)},
\end{eqnarray}
where $\phi$ is given by \eqref{phi} and 
\begin{eqnarray}\label{D}
D(b) =\sigma^2 b^2\gamma^-(b) + \mu(\theta b-(\theta-\eta)).
\end{eqnarray} 
The symmetric axis of the parabola $\Phi(b,\gamma)=0$ (see \eqref{polProp}) 
is $\displaystyle{\widehat{\gamma}(b)=-\frac{\mu(\theta b-(\theta-\eta))}{\sigma^2 b^2}}$; as a consequence, 
since $\gamma^-(b)$ is the negative solution to the equation $\Phi(b,\gamma)=0$, it follows $D(b)<0$ for each $b \in [0,1]$.
Since $\frac{\gamma^-(b)}{D(b)}>0$, $\phi(0)=\mu \theta>0$ and since, by some computations, 
\begin{eqnarray}\label{phiPrime}
\phi'(b) = -\frac{\sigma^2\mu(\theta-\eta)\gamma^-(b)}{\sigma^2 b^2\gamma^-(b) + \mu(\theta b-(\theta-\eta))}<0,
\end{eqnarray}
using $\eqref{gammaPrimeProp}$ we conclude that $\gamma^-$ admits in $[0,1]$ a unique minimizer $b^*>0$,
hence item (i) follows.

As for the second, notice that 
\begin{eqnarray}\label{IFF}
b^*<1 \quad \iff \quad (\gamma^-)'(1)>0.
\end{eqnarray}
Taking into account \eqref{gammaPrimeProp} and noting that  
(cf. \eqref{phi} and \eqref{D})
\begin{eqnarray*}
\phi(1)=\sigma^2\gamma^-(1)+\mu \theta
\end{eqnarray*}
and 
\begin{eqnarray*}
D(1)=\sigma^2 \gamma^-(1) + \mu\eta<0,
\end{eqnarray*}
we see that 
\begin{eqnarray}\label{expr2}
(\gamma^-)'(1)>0 \quad \iff\quad
\phi(1)=\sigma^2\gamma^-(1)+\mu \theta  =\mu\left(\theta - \eta -\sqrt{\eta^2+\frac{2\rho\sigma^2}{\lambda\mu^2}}\right) <0.
\end{eqnarray}
Hence, combining \eqref{IFF} and \eqref{expr2}, we get \eqref{cond:proportional}.
Finally, using again $\eqref{gammaPrimeProp}$, since $\phi(0)>0$, $\phi(1)<0$ and $\phi'(b)<0$ for each $b \in [0,1]$,
we conclude that $b^*$ is the unique solution of \eqref{equationPhi}.
\end{proof}

\section*{Proof of Proposition \ref{propEOL}}
\begin{proof}
From \eqref{pol}, by some computations, the derivative $(\gamma^-)'$ of $\gamma^-$ is
\begin{eqnarray}\label{gammaPrimeEOL}
(\gamma^-)'(b)=-\gamma^-(b)\frac{\phi(b)}{D(b)},
\end{eqnarray}
where 
\begin{eqnarray}
\phi(b) &:=& \frac{1}{2}\gamma^-(b)M_2'(b)+\theta M_1'(b), \label{gammaPrime:num}\\
D(b)&:=&M_2(b)\gamma^-(b) + \theta M_1(b) -(\theta-\eta)\mu. \label{gammaPrime:den}
\end{eqnarray} 
The symmetric axis of the parabola $\Phi(b,\gamma)=0$ (see \eqref{pol}) 
is $\displaystyle{\widehat{\gamma}(b)=-\frac{\theta M_1(b)-(\theta-\eta)\mu}{M_2(b)}}$; as a consequence, 
since $\gamma^-(b)$ is the negative solution to the equation $\Phi(b,\gamma)=0$, it follows $D(b)<0$ for each $b \in [0,1]$.
We compute
\begin{eqnarray*}
M_1'(b)&=&\frac{1}{(1-b)^2}\int_{\frac{b}{1-b}}^{+\infty} p(z)dz,\\
M_2'(b) &=& \frac{2b}{(1-b)^3}\int_{\frac{b}{1-b}}^{+\infty} p(z)dz,
\end{eqnarray*}
so that, by \eqref{gammaPrime:num},
\begin{eqnarray}\label{eqPhi}
\phi(b)=\psi(b)\chi(b),
\end{eqnarray}
where $\psi$ is given in \eqref{psi} 
and 
\begin{eqnarray*}
\chi(b):= \frac{1}{(1-b)^2}\int_{\frac{b}{1-b}}^{+\infty} p(z)dz, \quad b\in[0,1].
\end{eqnarray*}
Since $\chi$ is strictly positive on $[0,1)$ and $\lim_{b \to 1^-} \psi(b)=-\infty$, then by
\eqref{eqPhi} we obtain that there exists $\delta >0$ such that $\phi(b)<0$ for each $b \in (1-\delta,1)$. 
Hence, by \eqref{gammaPrimeEOL} and recalling that $\gamma^-(b)$ and $D(b)$ are strictly negative, 
we have $(\gamma^-)'(b)>0$ for each $b \in (1-\delta,1)$, which implies $1 \notin \mathcal B^*$.

Since $\psi(0)=\theta>0$ and recalling that $\lim_{b \to 1^-} \psi(b)=-\infty$, then the equation $\psi(b)=0$ has at least one solution in $(0,1)$, 
and $\gamma^-$ is therefore minimized by some $b^* \in (0,1)$ solution to equation \eqref{equationPsi}.
\end{proof}

\medskip

\indent \textbf{Acknowledgments.} 
The authors are grateful to Tingjin Yan for valuable comments on the first draft of the paper.
The work of Giorgio Ferrari has been funded by the Deutsche Forschungsgemeinschaft (DFG, German Research Foundation) – Project-
ID 317210226 – SFB 1283. Part of this work was done while Salvatore Federico and Maria Laura Torrente were visiting the Center for Mathematical Economics (IMW) at Bielefeld University. Both authors thank IMW for the hospitality


\end{document}